\documentclass[12pt]{amsart}

\usepackage{amsmath, amsthm, amssymb}
\input xy
\xyoption{all}
\usepackage{mathrsfs}
\usepackage{enumerate}
\usepackage{color}
\usepackage{amscd}

\newtheorem{dummy}{}[section]
\newtheorem{theorem}[dummy]{Theorem}

\newtheorem{lemma}[dummy]{Lemma}
\newtheorem{corollary}[dummy]{Corollary}

\theoremstyle{definition}
\newtheorem{definition}[dummy]{Definition}

\newtheorem{remark}[dummy]{Remark}

\newtheorem{result}{Theorem}
2

\newcommand{\nov}{\ensuremath{\mathrm{nov}}}
\newcommand{\FM}{\ensuremath{\mathbb{FM}}}
\newcommand{\N}{\ensuremath{\mathbb{N}}}
\newcommand{\Z}{\ensuremath{\mathbb{Z}}}
\newcommand{\Q}{\ensuremath{\mathbb{Q}}}
\newcommand{\C}{\ensuremath{\mathbb{C}}}
\renewcommand{\P}{\ensuremath{\mathbb{P}}}

\newcommand{\M}{\ensuremath{\overline{\mathcal{M}}}}
\renewcommand{\O}{\ensuremath{\mathcal{O}}}

\newcommand{\ev}{\ensuremath{\textrm{ev}}}
\newcommand{\vir}{\ensuremath{\textrm{vir}}}
\renewcommand{\t}{\ensuremath{\mathbf{t}}}

\newcommand{\ch}{\ensuremath{\textrm{ch}}}

\renewcommand{\d}{\ensuremath{\partial}}

\newcommand{\GIT}{\mkern-5mu\mathbin{/\mkern-6mu/\mkern-2mu}}
\newcommand{\vw}{\ensuremath{\vec{w}}}
\newcommand{\vd}{\ensuremath{\vec{d}}}
\newcommand{\Hom}{\ensuremath{\textrm{Hom}}}
\newcommand{\I}{\ensuremath{\mathcal{I}}}
\renewcommand{\H}{\ensuremath{\mathcal{H}}}

\newcommand{\m}{\ensuremath{\vec{m}}}
\newcommand{\V}{\ensuremath{\mathcal{V}}}

\newcommand{\cL}{\ensuremath{\mathcal{L}}}
\newcommand{\Contr}{\ensuremath{\textrm{Contr}}}
\newcommand{\Aut}{\ensuremath{\textrm{Aut}}}
\newcommand{\val}{\ensuremath{\textrm{val}}}
\newcommand{\bff}{\ensuremath{\mathbf{f}}}
\newcommand{\bfg}{\ensuremath{\mathbf{g}}}
\newcommand{\one}{\ensuremath{\mathbf{1}}}
\newcommand{\cT}{\ensuremath{\mathcal{T}}}
\newcommand{\bftau}{\ensuremath{\pmb{\tau}}}
\newcommand{\RC}{\ensuremath{\textrm{RC}}}

\newcommand{\T}{\ensuremath{\mathrm{T}}}
\newcommand{\ri}{\ensuremath{\textrm{i}}}
\newcommand{\re}{\ensuremath{\textrm{e}}}

\newcommand{\F}{\ensuremath{\mathcal{F}}}

\newcommand{\QM}{\ensuremath{\overline{\mathcal{QM}}}}

\newcommand{\ts}{\ensuremath{\tilde{s}}}

\newcommand{\bV}{\ensuremath{\mathbb{V}}}
\newcommand{\bU}{\ensuremath{\mathbb{U}}}
\newcommand{\un}{\ensuremath{\text{un}}}
\newcommand{\bt}{\ensuremath{\mathbf{t}}}

\begin{document}

\title[Sigma Models and Phase Transitions]{Sigma Models and Phase Transitions for Complete Intersections}
\author{Emily Clader and Dustin Ross}

\begin{abstract}
We study a one-parameter family of gauged linear sigma models (GLSMs) naturally associated to a complete intersection in weighted projective space. In the positive phase of the family we recover Gromov-Witten theory of the complete intersection, while in the negative phase we obtain a Landau--Ginzburg-type theory. Focusing on the negative phase, we develop foundational properties which allow us to state and prove a genus-zero comparison theorem that generalizes the multiple log-canonical correspondence and should be viewed as analogous to quantum Serre duality in the positive phase. Using this comparison result, along with the crepant transformation conjecture and quantum Serre duality, we prove a genus-zero correspondence between the GLSMs which arise at the two phases, thereby generalizing the Landau-Ginzburg/Calabi-Yau correspondence to complete intersections.
\end{abstract}

\maketitle

\section{Introduction}

In his seminal paper ``Phases of $N=2$ Theories in Two Dimensions'' \cite{Witten}, Witten introduced and studied a new type of supersymmetric quantum field theory known as the \emph{gauged linear sigma model} (GLSM).  Developed mathematically in recent work of Fan--Jarvis--Ruan \cite{FJRnew}, the GLSM depends on the choice of (1) a GIT quotient
\[X_{\theta} = [V \GIT_{\theta} G],\]
in which $V$ is a complex vector space, $G \subset \text{GL}(V)$, and $\theta$ is a character of $G$; and (2) a polynomial function $W: X_{\theta} \rightarrow \C$.  Witten conjectured that the GLSMs arising from different choices of $\theta$ should be related by analytic continuation, a relationship that he referred to as {\it phase transition}.

In particular, if $Z$ is a Calabi--Yau complete intersection in weighted projective space $\P(w_1, \ldots, w_M)$ defined by the vanishing of polynomials $F_1, \ldots, F_N$ of degrees $d_1,\ldots, d_N$, then there is a natural way to realize the Gromov--Witten theory of $Z$ as a GLSM with $V = \C^{M+N}$, $G=\C^*$, and
\[W = W(x_1, \ldots, x_M, p_1, \ldots, p_N) = \sum_{j=1}^N p_j F_j(x_1, \ldots, x_M).\]
Moreover, this model has two distinct phases, corresponding to the two distinct GIT quotients $X_+$ and $X_-$ of the form $[V\GIT_{\theta} G]$.  The positive phase of the GLSM yields the Gromov--Witten theory of $Z$, whereas the negative phase, which we refer to simply as the GLSM of $(X_-, W)$, yields a Landau--Ginzburg-type theory.

The main result of this paper is a genus-zero verification of Witten's proposal in this setting, under two additional assumptions:
\begin{enumerate}[(\bf{{A}}1)]
	\item for all $i$ and $j$, $w_i|d_j$;
	\item for any $m\in\Q/\Z$ such that $mw_{i_0} \in\Z$ for some $i_0 \in \{1, \ldots, M\}$, one has
	\[\#\{i  \; | \; mw_i \in\Z\} \geq \#\{j  \; | \; md_j \in\Z\}.\]
\end{enumerate}
We prove the following:
\begin{result}[see Theorem \ref{maintheorem2} for precise statement]
	When assumptions \textbf{(A1)} and \textbf{(A2)} are satisfied, the genus-zero Gromov--Witten theory of the complete intersection $Z$ can be explicitly identified with the genus-zero GLSM of $(X_-,W)$, after analytic continuation.
\end{result}

Witten's proposal has previously received a great deal of attention in the case where $Z$ is a hypersurface, in which case it is known as the Landau--Ginzburg/Calabi--Yau (LG/CY) correspondence. Mathematically, the LG/CY correspondence relates the Gromov--Witten theory of $Z$ to the quantum singularity theory of $W$, and it was proved in varying levels of generality by Chiodo--Ruan \cite{CR}, Chiodo--Iritani--Ruan \cite{CIR}, and Lee--Priddis--Shoemaker \cite{LPS}. It has been extended to the non-Calabi--Yau setting by Acosta \cite{Acosta} and Acosta--Shoemaker \cite{AcostaShoemaker}, and to very special complete intersections by the first author \cite{Clader}. 

Specializing the GLSM of $(X_-,W)$ to the hypersurface case does not immediately recover the quantum singularity theory as defined by Fan--Jarvis--Ruan \cite{FJR3, FJR, FJR2}. Rather, we recover a theory built from moduli spaces of weighted spin curves, which were introduced by the second author and Ruan in \cite{RR}. Nonetheless, the main result of \cite{RR} states that the theory built from weighted spin curves is equivalent, in genus zero, to the usual singularity theory, and it is through this equivalence that Theorem 1 generalizes the LG/CY correspondence.

Our proof of Theorem 1 is motivated by ideas introduced by Lee--Priddis--Shoemaker \cite{LPS}.  Specifically, for a particular action of $\T=(\C^*)^N$, we develop the following square of $\T$-equivariant equivalences:

\begin{equation}
	\label{diagram}
	\xymatrixcolsep{4pc}
	\xymatrix{
		\mathrm{GW}_{\T}(X_-) \ar@{<->}[r]^{\text{CTC}}\ar@{<->}[d]_{\text{Theorem 2}} &  \mathrm{GW}_{\T}(X_+)\ar@{<->}[d]^{\text{QSD}}\\
		\mathrm{GLSM}_{\T}(X_-,W) \ar@{<-->}[r] & \mathrm{GW}_\T(Z).
	}
\end{equation}

The right-hand vertical arrow is quantum Serre duality (QSD) and the quantum Lefschetz hyperplane principle, developed by Coates--Givental \cite{CoatesGivental} and Tseng \cite{Tseng}. It is an identification of a $\T$-equivariant extension of the Gromov--Witten theory of $Z$ to the $\T$-equivariant Gromov--Witten theory of $X_+$.  The top of the diagram is the crepant transformation conjecture (CTC), proved in this setting by Coates--Iritani--Jiang \cite{CIJ}.

The equivalence in the left-hand vertical arrow is new and is the technical heart of our paper.  It generalizes the multiple log-canonical (MLK) correspondence of \cite{LPS} and serves as the analogue in the negative phase of quantum Serre duality:

\begin{result}[see Theorem \ref{maintheorem} for precise statement]
	A $\T$-equivariant extension of the narrrow, genus-zero GLSM of $(X_-,W)$ can be explicitly identified with the $\T$-equivariant genus-zero Gromov--Witten theory of $X_-$.
\end{result}

From here, Theorem 1 follows by carefully taking the non-equivariant limit of the composition in diagram \eqref{diagram}. Although our proof of Theorem 2 does not require (\textbf{A1}) and (\textbf{A2}) to hold, the existence of the non-equivariant limit does require these additional assumptions. These assumptions generalize the ``Fermat-type'' condition in the hypersurface case, which is the only setting in which the LG/CY correspondence is currently known.  

We expect Theorem 2 to hold in much greater generality. In particular, our methods should generalize to prove Theorem 2 for a large class of GLSMs with strong torus actions. Our restriction to complete intersections was mostly for pedagogical reasons. In particular, it provides a natural class of GLSMs where the diagram \eqref{diagram} can be made explicit without cluttering the results with an overabundance of notation.

\subsection{Precise statements of results}
\label{precise}

As above, let 
\[Z=Z(F_1,\dots,F_N) \subset \P(w_1, \ldots, w_M)\]
be a complete intersection in weighted projective space defined by the vanishing of quasihomogeneous polynomials $F_1, \ldots, F_N$ of degrees $d_1, \ldots, d_N$.  (We do not yet require that $Z$ be Calabi--Yau.)  Let $G = \C^*$ act on $V=\C^{M+N}$ with weights
\[(w_1, \ldots, w_M, -d_1, \ldots, -d_N).\]
We denote the coordinates on $V$ by $(x_1, \ldots, x_M, p_1, \ldots, p_N)$.

There are two GIT quotients of the form $[V\GIT_{\theta} G]$, depending on whether $\theta \in \Hom(G,\C^*) \cong \Z$ is positive or negative.  Explicitly, these are
\[X_+ := \bigoplus_{j=1}^N \O_{\P(w_1, \ldots, w_M)}(-d_j)\]
and
\[X_- :=\bigoplus_{i=1}^M \O_{\P(d_1, \ldots, d_N)}(-w_i).\]
For each of these choices, one obtains a GLSM with superpotential $W= \sum_j p_jF_j(x_1, \ldots, x_M)$.

Most of this paper concerns the negative phase.  In this case, the GLSM consists of the following basic ingredients, which we describe explicitly in Section \ref{GLSM}:
\begin{enumerate}
\item a ``narrow state space", which is a vector subspace \[\H^W \subset H_{CR}^*(X_-);\]
\item a moduli space $\QM^W_{g,n}(X_-,\beta)$ equipped with cotangent line classes 
\[
\psi_i\in H^2(\QM^W_{g,n}(X_-,\beta))
\]
and evaluation maps 
\[
\ev_i: \QM^W_{g,n}(X_-,\beta) \rightarrow \I X_-
\]
for $i\in\{1,\dots,n\}$, and a decomposition into ``broad" and ``narrow" components;
\item a virtual fundamental class 
\[
[\QM^W_{g,n}(X_-,\beta)]^{\vir}\in H_*\left(\QM^W_{g,n}(X_-,\beta)\right)
\]
defined for all narrow components of the moduli space.
\end{enumerate}

For any choice $\phi_1,\dots,\phi_n\in\H^W$, these ingredients can be combined to define narrow GLSM correlators
\[
\langle \phi_1 \psi^{a_1} \cdots \phi_n \psi^{a_n}\rangle^{X_-, W}_{g,n,\beta}:=\int\limits_{[\QM^W_{g,n}(X_-,\beta)]^{\vir}} \ev_1^*(\phi_1)\psi_1^{a_1} \cdots \ev_n^*(\phi_n)\psi_n^{a_n}\in\Q.
\]

In genus zero, using a natural action of the torus $\T=(\C^*)^N$, the basic ingredients can be extended $\T$-equivariantly, and we define extended GLSM correlators for {\it any} cohomology insertions $\phi_1,\dots,\phi_n\in H_{CR}^*(X_-)$, denoted
\[
\langle \phi_1 \psi^{a_1} \cdots \phi_n \psi^{a_n}\rangle^{X_-, W,\T}_{0,n,\beta}.
\]
The extended correlators are useful in the statements and proofs of our main results, since they allow us to consider all of $H^*_{CR}(X_-)$ as a state space. They should not be confused, however, with the so-called ``broad insertions" in quantum singularity theory.

In analogy, the $\T$-equivariant Gromov--Witten (GW) invariants of $X_-$, denoted
\[
\langle \phi_1 \psi^{a_1} \cdots \phi_n \psi^{a_n}\rangle^{X_-,\T}_{g,n,\beta},
\]
are also defined for $\phi_1,\dots,\phi_n\in H_{CR}^*(X_-)$. Since GW invariants satisfy the string equation, dilaton equation, and topological recursion relations, the genus-zero GW invariants can be placed in the language of Givental's formalism. In particular, there is an infinite-dimensional symplectic vector space $\V_\T^{X_-}$ and an over-ruled Lagrangian cone 
\[
\cL_\T^{X_-}\subset\V_\T^{X_-}
\] 
that encodes all genus-zero GW invariants of $X_-$. More specifically, a formal germ $\hat\cL_\T^{X_-}$ of $\cL_\T^{X_-}$ is given by points of the form
\[
 -\mathbf{1}z + \bftau(z) +\sum_{n,\beta,\mu} \frac{Q^\beta}{n!}\left\langle \bftau(\psi)^n\; \frac{\Phi_\mu}{-z-\psi}\right\rangle^{X,\T}_{0,n+1,\beta} \Phi^\mu
\]
(see Section \ref{GWtheory} for notation and precise definitions). 

To compare the GLSM invariants of $(X_-,W)$ with the GW invariants of $X_-$, we also encode the GLSM invariants in $\V_\T^{X_-}$. Specifically, we define $\hat\cL_\T^{X_-,W}$ to be the formal subspace consisting of points of the form
\[
I_\T^{X_-,W}(Q,-z)+\bt(z) +\sum_{n,\beta,\mu} \frac{Q^\beta}{n!}\left\langle \bt(\psi)^n\; \frac{\Phi_\mu}{-z-\psi}\right\rangle^{X_-,W,\T}_{0,n+1,\beta} \Phi^\mu,
\]
where $I_\T^{X_-,W}$ is an explicit cohomology-valued series. Since the GLSM invariants do not, a priori, satisfy the string and dilaton equations, $\hat\cL_\T^{X_-,W}$ does not necessarily have the same geometric properties as its GW analogue. However, our first theorem states that the formal subspace $\hat\cL_\T^{X_-,W}$ encodes all of the structure of the GW over-ruled Lagrangian cone $\cL_\T^{X_-}$.

\begin{theorem}\label{maintheorem}
With notation as above, $\hat\cL_\T^{X_-,W}$ is a formal germ of the over-ruled Lagrangian cone $\cL_\T^{X_-}$.
\end{theorem}

A consequence of Theorem \ref{maintheorem} is that all extended genus-zero GLSM invariants of $(X_-,W)$ are determined from the equivariant genus-zero GW invariants of $X_-$, and vice versa.  In light of this, we denote $\cL_\T^{X_-,W} = \cL_\T^{X_-}$.

When $N=1$, as mentioned above, the GLSM of $(X_-,W)$ is equivalent to the quantum singularity theory of the polynomial $F = F_1$, and in this case, Theorem \ref{maintheorem} is related to the MLK correspondence  of Lee--Priddis--Shoemaker.  Moreover, since $\text{GLSM}(X_+, W)=\text{GW}(Z)$, Theorem \ref{maintheorem} can be viewed as an analogue in the negative phase of quantum Serre duality, which gives a symplectomorphism 
$\phi_\T^+:\mathcal{V}_{\T}^{X_+}\stackrel{\sim}{\rightarrow}\mathcal{V}_{\T}^{Z}$ such that $\phi_\T^+\left(\cL_\T^{X_-}\right)=\cL_\T^{Z}$.  Here, $\cL_\T^Z$ is the Lagrangian cone associated to the $\T$-equivariant ambient GW theory of $Z$.

Thus equipped with GW/GLSM correspondences relating the two theories associated to $(X_{\theta}, W)$ for each fixed $\theta$, what remains is to study how these theories change when $\theta$ varies.  For this, we will assume the Calabi--Yau condition:
\[\sum_{i=1}^M w_i = \sum_{j=1}^N d_j.\]

In the GW case, the connection between the GW theory of $X_+$ and $X_-$ is the subject of Ruan's crepant transformation conjecture (CTC) \cite{CoatesRuan}, which, roughly speaking, states that there is a symplectomorphism $\bU_\T:\V_\T^{X_-}\stackrel{\sim}{\rightarrow} \V_\T^{X_+}$ such that $\bU_{\T}\left(\cL_\T^{X_-}\right)=\cL_\T^{X_+}$ after analytic continuation.\footnote{The analytic continuation should be understood in the following sense. The toric mirror theorem of Coates--Corti--Iritani--Tseng \cite{CCIT2} provides $I$-functions for each GIT phase which entirely determine the GW Lagrangian cones. The $I$-functions depend on local K\"ahler parameters, and Coates--Iritani--Jiang \cite{CIJ} construct an explicit global K\"ahler moduli space along which the $I$-functions can be analytically continued from one phase to the other.} In our particular situation, the CTC is a very special application of the far-reaching toric CTC proved by Coates--Iritani--Jiang \cite{CIJ}. In addition to computing $\bU_\T$, Coates--Iritani--Jiang prove that it is induced by a Fourier--Mukai transformation (c.f. Section \ref{transitions}).

When combined with quantum Serre duality and the CTC, Theorem \ref{maintheorem} gives the analogous phase transition for the extended GLSM.  More precisely, it implies that there is a $\T$-equivariant symplectomorphism $\bV_\T:\V_\T^{X_-}\rightarrow\V_\T^Z$ such that 
\[
\bV_\T\left(\cL_\T^{X_-,W}\right) = \cL_\T^{Z},
\] 
after analytic continuation.

In order to use this statement to relate the non-extended theories, we must verify that it is possible to take a meaningful non-equivariant limit of $\bV_\T$, after restricting to the narrow subspace.  The existence of such non-equivariant limits in general is a subtle question. We prove in Section \ref{transitions} that, under assumptions \textbf{(A1)} and \textbf{(A2)}, we can take a non-equivariant limit, and the resulting symplectomorphism identifies the non-equivariant cones $\cL^{X_-,W}$ and $\cL^Z$ associated to the narrow GLSM and the ambient GW theory of $Z$, respectively.

\begin{theorem}
	\label{maintheorem2}
	Under the assumptions (\textbf{A1}) and (\textbf{A2}), there exists a symplectomorphism $\bV:\V^{X_-,W}\stackrel{\sim}{\rightarrow} \V^{Z}$ that identifies the Lagrangian cones after analytic continuation and specializing $Q=1$:
	\[
	\bV\left(\cL^{X_-,W}\right)=\cL^Z.
	\]
\end{theorem}

\subsection{Plan of the paper}

We begin, in Section \ref{GLSM}, by describing the definition and basic properties of the genus-zero GLSM of $(X_-,W)$, following Fan--Jarvis--Ruan \cite{FJRnew}. In Section \ref{GWtheory}, we introduce GW theory and Givental's formalism, allowing us to precisely define the Lagrangian cones and formal subspaces that appear in the statements of Theorems \ref{maintheorem} and \ref{maintheorem2}.

The proof of Theorem \ref{maintheorem} is contained in Sections \ref{localization}, \ref{hybrid}, and \ref{proof}. The key idea of the proof is to express the GLSM invariants, via virtual localization, in terms of graph sums. The graph sum expression is written explicitly in Section \ref{localization}. Each vertex in a localization graph contributes an integral over a moduli space of weighted spin curves, and we discuss these integrals in Section \ref{hybrid}. In particular, we prove a vertex correspondence comparing these integrals to twisted GW invariants of an orbifold point. In Section \ref{proof}, we prove a characterization of points on the formal subspace $\hat\cL_{\T}^{X_-,W}$ in terms of a recursive structure on graph sums obtained by removing edges. We use this characterization, along with the vertex correspondence, to conclude the proof of Theorem \ref{maintheorem}.

Theorem \ref{maintheorem2} is proved in Section \ref{transitions} by carefully studying the existence of non-equivariant limits.

\subsection{Acknowledgments}

The authors thank Pedro Acosta for many enlightening conversations in the early days of this project, and Mark Shoemaker for carefully explaining several technical aspects of \cite{LPS}.  Thanks are also due to Yongbin Ruan for introducing the authors to the concepts and constructions of the gauged linear sigma model.  The first author acknowledges the generous support of Dr.~Max R\"ossler, the Walter Haefner Foundation, and the ETH Foundation.  The second author has been supported by NSF RTG grants DMS-0943832 and DMS-1045119 and the NSF postdoctoral research fellowship DMS-1401873.

\section{Definitions and Setup}
\label{GLSM}

The general construction of the GLSM was given by Fan--Jarvis--Ruan in \cite{FJRnew}.  In this section, we discuss the special class of GLSMs which arise in the study of complete intersections.

\subsection{Input data}

Fix a vector space $V = \C^{M+N}$ with coordinates $(x_1, \ldots, x_M, p_1, \ldots, p_N)$, and choose positive integers $w_1, \ldots, w_M$ and $d_1, \ldots, d_N$.  For each $j \in \{1, \ldots, N\}$, let $F_j(x_1, \ldots, x_M)$ be a quasihomogeneous polynomial of weights $(w_1, \ldots, w_M)$ and degree $d_j$; that is,
\[
F_j(\lambda^{w_1}x_1, \ldots, \lambda^{w_M}x_M) = \lambda^{d_j} F_j(x_1, \ldots, x_M)
\]
for any $\lambda \in \C$. We assume $\gcd(w_1,\dots,w_M,d_1,\dots,d_N)=1$.

Each equation $\{F_j = 0\}$ defines a hypersurface in the weighted projective space $\P(\vw) = \P(w_1, \ldots, w_M)$.  Assume that the $F_j$ are {\it nondegenerate} in the sense that (1) the hypersurfaces $\{F_j  = 0\}$ are all smooth and (2) they intersect transversally.

Let
\[G := \{(\lambda^{w_1}, \ldots, \lambda^{w_M}, \lambda^{-d_1}, \ldots, \lambda^{-d_N})\; | \; \lambda \in \C^*\} \cong \C^*\]
act diagonally on $V$.  For a nontrivial character $\theta \in \Hom_{\Z}(G, \C^*)\cong \Z$, one obtains a GIT quotient $X_{\theta} = [V \GIT_{\theta} G]$.  Until stated otherwise, we will always take $\theta < 0$, so the resulting GIT quotient is
\begin{equation}\label{GITpresentation}
X:=X_-= \bigoplus_{i=1}^M \O_{\P(\vd)} (-w_i).
\end{equation}
(Observe that we denote $X_-$ simply by $X$ in what follows, to ease the notation.)

The {\it superpotential} of the theory is the function $W: X \rightarrow \C$ defined by
\[W(x_1, \ldots, x_M, p_1, \ldots, p_N) := \sum_{j=1}^N p_j F_j(x_1, \ldots, x_M).\]

\subsection{State space}

The state space of the GLSM is, by definition, the vector space
\[H^*_{CR}(X, W^{+\infty}; \C),\]
where $W^{+ \infty}$ is a Milnor fiber of $W$, defined by $W^{+ \infty} = W^{-1}(M)$ for a sufficiently large real number $M$.

This state space has summands indexed by the components of the inertia stack $\I X$, which, in turn, are labeled by elements $g \in G$ with nontrivial fixed-point set $\text{Fix}(g)\subset X$.  A summand indexed by $g$ is said to be {\it narrow} if $\text{Fix}(g)$ is compact; otherwise, the summand is {\it broad}.

Let us describe the narrow part of the state space more concretely.  The only $g=(\lambda^{w_1}, \ldots, \lambda^{w_M}, \lambda^{-d_1}, \ldots, \lambda^{-d_N})$ with nontrivial fixed-point set are those for which $\lambda^{d_j} =1$ for some $j$. In particular, $\lambda=\re^{2\pi\ri m}$ for some $m\in \Q/\Z$.  Furthermore, elements $(\vec{x}, \vec{p}) \in \text{Fix}(g)$ must have $x_i = 0$ whenever $\lambda^{w_i} \neq 1$, but there is no constraint on the $x_i$ for which $\lambda^{w_i} =1$.  As a result, $\text{Fix}(g)$ is compact exactly when there is no $i$ such that $\lambda^{w_i} =1$, and in this case, it equals
\begin{equation}
	\label{Plambda}
	X_{(m)} :=\left\{p_j = 0 \text{ for all } j \text{ with } md_j\notin\Z\right\} \subset \P(\vd)\subset X.
\end{equation}
Furthermore, for such $g$, we have $W|_{\text{Fix}(g)} \equiv 0$, so the relative cohomology group in the definition of the state space restricts to an absolute cohomology group.

To summarize, we have the following:

\begin{definition}
	Let
	\[\text{nar} := \left\{ m\in\Q/\Z \; \bigg | \; \exists \; j \text{ such that }md_j\in\Z, \; \not \exists \; i \text{ such that } mw_i\in\Z\right\}.\]
	The {\it narrow state space} is
	\[\mathcal{H}^W:= \bigoplus_{m \in \text{nar}} H^*(X_{(m)})\subset H_{CR}^*(X),\]
	where $X_{(m)}$ is defined as in (\ref{Plambda}).
\end{definition}

\subsection{Moduli space}

Fix a genus $g$, a degree $\beta \in \Q$, and a nonnegative integer $n$.

\begin{definition}
	\label{LGmap}
	A {\it stable Landau--Ginzburg quasi-map} to the pair $(X,W)$ consists of an $n$-pointed prestable orbifold curve $(C; q_1, \ldots, q_n)$ of genus $g$, an orbifold line bundle $L$ of degree $\beta$ on $C$, and a section
	\[\sigma \in \Gamma\left(\bigoplus_{i=1}^M L^{\otimes w_i} \oplus \bigoplus_{j=1}^N (L^{\otimes - d_j} \otimes \omega_{C,\log})\right),\]
	where
	\[\omega_{C,\log} := \omega_C([q_1] + \cdots + [q_n]).\]
	We denote the components of $\sigma$ by
	\[\sigma = (x_1, \ldots, x_M, p_1, \ldots, p_N).\]
	We require that this data satisfies the following conditions:
\begin{enumerate}
\item \emph{Nondegeneracy} (for $\theta<0$): The points $q\in C$ satisfying $p_1(q)=\dots=p_N(q)=0$ are finite and disjoint from the marks and nodes of $C$. We refer to such points as \emph{basepoints}.
\item \emph{Representability}: For every $q \in C$ with isotropy group $G_q$, the homomorphism $G_q \rightarrow \C^*$ giving the action of the isotropy group on the bundle $\bigoplus_i L^{\otimes w_i} \oplus \bigoplus_j L^{\otimes -d_j}$ is injective.
\item \emph{Stability}: $(L^\vee\otimes\omega_{\log})^{\otimes \epsilon}\otimes\omega_{\log}$ is ample for all $\epsilon>0$.
\end{enumerate}

	A {\it morphism} between $(C;q_1, \ldots, q_n; L;\sigma)$ and $(C'; q_1', \ldots, q_n'; L', \sigma')$ consists of a morphism $s: C \rightarrow C'$ such that $s(q_i) = q_i'$, together with a morphism $s^*L' \rightarrow L$ which, in combination with the natural isomorphism $s^*\omega_{C', \log} \xrightarrow{\sim} \omega_{C, \log}$, sends $s^*\sigma'$ to $\sigma$.
\end{definition}

This stability condition is referred to as ``$\epsilon = 0$" stability in the language of \cite{FJRnew}.  In particular, it prohibits the curve $C$ from having rational tails (genus-zero components with only one special point), and it imposes that $\beta < 0$ on any genus-zero component with exactly two special points.

Fan--Jarvis--Ruan prove in \cite{FJRnew} that there is a finite-type, separated, Deligne--Mumford stack $\QM^W_{g,n}(X,\beta)$ parameterizing families of stable Landau--Ginzburg quasi-maps to $(X,W)$ up to isomorphism. This stack admits a perfect obstruction theory
\begin{equation}
	\label{pot}
	\mathbb{E}^{\bullet} = R^\bullet\pi_*\left(\bigoplus_{i=1}^M \mathcal{L}^{\otimes w_i} \oplus \bigoplus_{j=1}^N (\mathcal{L}^{\otimes -d_j} \otimes \omega_{\pi, \log})\right)^{\vee}
\end{equation}
relative to the Artin stack $\mathcal{D}_{g,n,\beta}$ of $n$-pointed genus-$g$ orbifold curves with a degree-$\beta$ line bundle $L$.  Here,
\[
\pi: \mathcal{C} \rightarrow \QM^W_{g,n}(X,\beta)
\]
denotes the universal family and $\mathcal{L}$ the universal line bundle on $\mathcal{C}$.

\subsection{Evaluation maps}

Certain substacks of $\QM^W_{g,n}(X,\beta)$ will be particularly important in what follows.  To define them, recall that if $q$ is a point on an orbifold curve $C$ with isotropy group $\Z_r$ and $L$ is an orbifold line bundle on $C$, then the {\it multiplicity} of $L$ at $q$ is defined as the number $m \in \Q/\Z$ such that the canonical generator of $\Z_r$ acts on the total space of $L$ in local coordinates near $q$ by
\[(x,v) \mapsto (\re^{2\pi \ri \frac{1}{r}} x, \re^{2\pi \ri m} v).\]
For $\m=(m_1,\dots,m_n)$, we define
\[
\QM^W_{g,\m}(X,\beta) \subset \QM^W_{g,n}(X,\beta)
\]
as the open and closed substack consisting of elements for which the multiplicity of $L$ at $q_i$ is equal to $m_i$.  Note that if
\[d := \text{lcm}(d_1, \ldots, d_N),\]
then $m_i \in \frac{1}{d}\Z/\Z$.

To define evaluation maps, let $\varsigma$ be the universal section of the universal bundle $\bigoplus_i \mathcal{L}^{\otimes w_i} \oplus \bigoplus_j (\mathcal{L}^{\otimes - d_j} \otimes \omega_{\pi,\log})$ on $\mathcal{C}$.  If $\Delta_k \subset \mathcal{C}$ denotes the divisor corresponding to the $k$th marked point, then
\[
\varsigma|_{\Delta_k} \in \Gamma\left(\bigoplus_{i=1}^M \mathcal{L}^{\otimes w_i} \oplus \bigoplus_{j=1}^N \mathcal{L}^{\otimes -d_j}\bigg|_{\Delta_k}\right),
\]
using the fact that $\omega_{\pi, \log}|_{\Delta_k}$ is trivial.  Furthermore, the image of this section must be zero on any component on which the action of the isotropy group is nontrivial.  It follows that $\sigma(q_k)$ defines an element of $X_{(m_k)}$.  Thus, we can define evaluation maps to the inertia stack:
\[
\ev_i: \QM^W_{g,\m}(X,\beta) \rightarrow \I X
\]
\[
(C; q_1, \ldots, q_n; L; \sigma) \mapsto \sigma(q_k)\in X_{(m_k)}.
\]

\subsection{Virtual cycle}

The definition of the virtual cycle in the GLSM relies on the cosection technique of Kiem--Li \cite{KL}, which was first applied in this setting by Chang--Li--Li \cite{CL, CLL}.  For the specific case of the GLSM, we refer the reader to \cite{FJRnew} for details. The construction is quite subtle; in particular, a cosection can only be defined on components $\QM^W_{g,\m}(X,\beta)$ for which $m_k \in \text{nar}$ for every $k$. For such a component, consider the substack of $\QM^{W}_{g,\m}(X,\beta)$ consisting of sections $\sigma$ with image in the critical locus of $W$.  By our non-degeneracy assumptions on the polynomials $F_j$, this is equivalent to the requirement that the first $M$ components of $\sigma$ vanish, and we define
\[
\overline{\mathcal{QM}}^W_{g,\m}\left(\P(\vec d),\beta\right):=\{(C,L,\sigma)\;|\; x_i=0 \; \forall i\}.
\]
The cosection technique provides a virtual cycle on this substack: 
\[
\left[\QM^{W}_{g,\m}(X,\beta)\right]^{\vir}\in H_*\left(\QM^{W}_{g,\m}\left(\P(\vec d),\beta\right)\right).
\]
A key result about the stack of sections supported on the critical locus is the following.

\begin{theorem}[Fan--Jarvis--Ruan \cite{FJRnew}]\label{thm:properness}
The substack $\overline{\mathcal{QM}}^W_{g,\m}(\P(\vec d),\beta)$ is proper.
\end{theorem}

Integrals against the virtual fundamental class are expressed as follows:

\begin{definition}
	Given
	\[\phi_1, \ldots, \phi_n \in \mathcal{H}^W\]
	and nonnegative integers $a_1, \ldots, a_n$, the associated {\it degree-$\beta$, genus-$g$ GLSM correlator} is defined by
	\[
	\langle \phi_1 \psi^{a_1} \cdots \phi_n \psi^{a_n}\rangle^{X, W}_{g,n,\beta} =
	\int_{[\QM^W_{g,n}(X,\beta)]^{\vir}} \ev_1^*(\phi_1)\psi_1^{a_1} \cdots \ev_n^*(\phi_n)\psi_n^{a_n},
	\]
	where $\psi_i$ is the cotangent line class at the $i$th marked point on the coarse curve.
\end{definition}

Note that if $\phi_k$ is drawn from the component of $\mathcal{H}^W$ indexed by $m_k$, then the above integral is supported on $\QM^W_{g,\m}(X,\beta)$.  Thus, the definition of the correlators makes sense even though a virtual cycle has only been defined on the narrow substacks of $\QM^W_{g,n}(X,\beta)$.

\subsection{Genus zero}

In genus zero, the cosection construction is not needed.  The key point is the following:

\begin{lemma}
	\label{sm}
	Let $\m$ be such that $m_k \in \text{nar}$ for every $k$.  Then
\[
\QM^W_{0,\m}(X,\beta)=\QM^W_{0,\m}(\P(\vec d),\beta).
\] 
In particular, $\QM^W_{0,\m}(X,\beta)$ is proper.
	\begin{proof}
		It suffices to prove that $H^0(C,L^{\otimes w_i})=0$ for each $(C,L,\sigma)\in\QM^W_{0,\m}(X,\beta)$. Let $\sigma = (x_1, \ldots, x_M, p_1, \ldots, p_N)$.  Since the $p_j$ cannot simultaneously vanish everywhere, at least one of the bundles $L^{\otimes - d_j} \otimes \omega_{\log}$ must have nonnegative degree, so
		\[-d_j \beta - 2 + n \geq 0.\]
		Thus, using the fact that $w_i\leq d_j$ for all $i,j$, we have
		\[\deg(L^{\otimes w_i}) = w_i \beta \leq \frac{w_i}{d_j}(n-2) < n-1.\]
		
		On the other hand, the condition that $m_k \in \text{nar}$ means that the isotropy group at $q_k$ acts nontrivially on the fiber of $L^{\otimes w_i}$ for each $i$ and $k$, so the sections $x_i$ must vanish at all $n$ marked points.  When $C$ is smooth, it follows that $x_i \equiv 0$.  More generally, the same argument as above shows that on each irreducible component $C' \subset C$, we have $\deg(L^{\otimes w_i}|_{C'}) < n' -1$, where $n'$ is the number of marks and nodes on $C'$.  An inductive argument on the number of components then implies, again, that $x_i \equiv 0$.
	\end{proof}
\end{lemma}

It follows that the cosection-localized virtual cycle on $\QM^W_{0,\m}(X,\beta)$ is simply the usual virtual cycle, defined via the perfect obstruction theory (\ref{pot}).  To put it more explicitly, one can define 
$[\QM^W_{0,\m}(\P(\vec d),\beta)]^{\vir}$ by way of the perfect obstruction theory $\bigoplus_jR\pi_*(\mathcal{L}^{\otimes -d_j}\otimes\omega_{\pi,\log})^{\vee}$ and then cap with the top Chern class of an obstruction bundle:
\begin{equation}\label{virclass}
[\QM^W_{0,\m}(X,\beta)]^{\vir} = e\left(\bigoplus_{i=1}^MR^1\pi_*\left( \mathcal{L}^{\otimes w_i}\right) \right) \cap [\QM^W_{0,\m}(\P(\vec d),\beta)]^{\vir}.
\end{equation}

\begin{remark}\label{rmk:onebroad}
The same proof as that given in Lemma \ref{sm} shows that $H^0(C, L^{\otimes w_i}) = 0$ if all but one marked point is narrow, so \eqref{virclass} remains valid in this case.
\end{remark}

\subsection{Extended GLSM theory}

For our methods, it is useful to extend the definition of the GLSM invariants beyond the narrow state space. This can be done by working equivariantly. More specifically, let $\T=(\C^*)^N$ act on $\C^{M+N}$ by
\[
(t_1,\dots,t_N)\cdot(x_1,\dots,x_M,p_1,\dots,p_N):=(x_1,\dots,x_M,t_1p_1,\dots,t_Np_N).
\]
Then $\T$ acts on $\QM^W_{0,n}(X,\beta)$ by post-composing the section $\sigma$ with the action of $\T$. The fixed loci of this action lie in the critical locus $\QM^W_{0,n}(\P(\vec d),\beta)$ (even when the marked points are not narrow). 

The $\T$-action on $\QM^W_{g,n}(X,\beta)$ induces a canonical lift to a $\T$-action on $\bigoplus_{i=1}^M R\pi_*(\mathcal{L}^{\otimes w_i})$. In analogy with \eqref{virclass}, we define a $\T$-equivariant extended virtual class by
\[
[\QM^W_{g,n}(X,\beta)]_\T^{\vir}:=e_{\T}^{-1}\left(\bigoplus_{i=1}^M R\pi_*(\mathcal{L}^{\otimes w_i}) \right)\cap [\QM^W_{g,n}(\P(\vec d),\beta)]^{\vir}.
\]

\begin{definition}
	Given
	\[
	\phi_1, \ldots, \phi_n \in \H:= H_{CR,\T}^*(X)
	\]
	and nonnegative integers $a_1, \ldots, a_n$, the associated {\it degree-$\beta$, genus-$g$, extended GLSM correlator} is defined by
	\[
	\langle \phi_1 \psi^{a_1} \cdots \phi_n \psi^{a_n}\rangle^{X, W,\T}_{g,n,\beta} :=
	\int_{[\QM^W_{g,n}(X,\beta)]_{\T}^{\vir}} \ev_1^*(\phi_1)\psi_1^{a_1} \cdots \ev_n^*(\phi_n)\psi_n^{a_n}.
	\]
	These invariants lie in $\C(\alpha)$, where $\alpha=(\alpha_1,\dots,\alpha_N)$ are the equivariant parameters for the $\T$-action; that is, $H^*(\mathcal{B}\T) = \C[\alpha]=\C[\alpha_1,\dots,\alpha_N]$.
\end{definition}

If all of the insertions come from the narrow state space $\H^W \subset \H$, then the genus-zero extended correlators admit a non-equivariant limit $\alpha \rightarrow 0$, which recovers the definition of narrow GLSM correlators.

\subsection{Fixed-point basis}

The state space $\H$ has a special basis given by the localization isomorphism:
\[
H_{CR,\T}^*(X)\otimes\C(\alpha)\cong\bigoplus_{k=1}^N H_{CR,\T}^*(P_k)\otimes\C(\alpha).
\] 
Here, $P_k$ is the unique $\T$-fixed point of $X$ where $p_k\neq 0$.  For $m \in \Q/\Z$, we denote by $\one_{(m)}^k$ the fundamental class on the twisted sector of $H_{CR,\T}^*(P_k)$ indexed by $m$.  The collection $\{\one_{(m)}^k\}$ is referred to as the \emph{fixed-point basis} of $\H$.

\section{Gromov--Witten Theory and Lagrangian Cones}
\label{GWtheory}

We now provide a definition of stable maps and GW invariants that is notationally consistent with the definition of LG stable quasi-maps given above, and we describe Givental's axiomatic framework for genus-zero GW theory.

\begin{definition}
	\label{stablemap}
	A {\it stable map} to $X$ consists of an $n$-pointed prestable orbifold curve $(C; q_1, \ldots, q_n)$ of genus $g$, an orbifold line bundle $L$ of degree $\beta$ on $C$, and a section
	\[\sigma \in \Gamma\left(\bigoplus_{i=1}^M L^{\otimes w_i} \oplus \bigoplus_{j=1}^N L^{\otimes - d_j}\right).\]
	We require that this data satisfies the following conditions:
\begin{enumerate}
\item \emph{Nondegeneracy} (for $\theta<0$): There are no points $q\in C$ satisfying $p_1(q)=\dots=p_N(q)=0$.
\item \emph{Representability}: For every $q \in C$ with isotropy group $G_q$, the homomorphism $G_q \rightarrow \C^*$ giving the action of the isotropy group on the bundle $\bigoplus_i L^{\otimes w_i} \oplus \bigoplus_j L^{\otimes -d_j}$ is injective.
\item \emph{Stability}: $(L^\vee)^{\otimes \epsilon}\otimes\omega_{\log}$ is ample for all $\epsilon\gg0$.
\end{enumerate}

A {\it morphism} between $(C;q_1, \ldots, q_n; L;\sigma)$ and $(C'; q_1', \ldots, q_n'; L', \sigma')$ consists of a morphism $s: C \rightarrow C'$ such that $s(q_i) = q_i'$, together with a morphism $s^*L' \rightarrow L$ that sends $s^*\sigma'$ to $\sigma$.
\end{definition}

\begin{remark}
By our conventions, the degree $\beta$ is non-positive. This convention is consistent with the fact that we are working in the negative chamber of the GIT quotient.
\end{remark}

\begin{remark}
In contrast to the definition of stable Landau--Ginzburg quasi-maps,	the stability condition in Definition \ref{stablemap} is ``$\epsilon = \infty$" stability, which is equivalent to the requirement that $(C; q_1, \ldots, q_n; L)$ have finitely many automorphisms.  In particular, rational tails are {\it not} prohibited.
	
Perhaps a more natural analogue of Definition \ref{LGmap} is that of stable quasi-maps, developed by Ciocan-Fontanine--Kim--Maulik \cite{CFKM}. However, we choose to work with stable maps because Givental's axiomatic framework is more natural in this setting. Using the wall-crossing results of Ciocan-Fontanine--Kim \cite{CFK2}, one could reprove our results using quasi-maps, instead. 
\end{remark}

As is well-known, the moduli spaces $\M_{g,n}(X,\beta)$ of stable maps are finite-type, separated, Deligne--Mumford stacks. When $\beta=0$, the moduli stacks are not necessarily proper. To remedy the nonproperness, consider the $\T$-action on $\M_{g,n}(X,\beta)$ defined by postcomposing each stable map with the $\T$-action on $X$. As in the GLSM setting, the fixed loci of the $\T$-action lie in $\M_{g,n}(\P(\vec d),\beta)$ and there is a canonical lift of the $\T$-action to $\bigoplus_{i=1}^M R\pi_*(\mathcal{L}^{\otimes w_i})$. We define a $\T$-equivariant virtual class by
\[
[\M_{g,n}(X,\beta)]_\T^{\vir}=e_{\T}^{-1}\left(\bigoplus_{i=1}^M R\pi_*(\mathcal{L}^{\otimes w_i}) \right)\cap [\M_{g,n}(\P(\vec d),\beta)]^{\vir}.
\]

The moduli spaces $\M_{g,n}(X,\beta)$ admit natural evalution maps to $\H$, and we have the following definition of $\T$-equivariant GW correlators.

\begin{definition}\label{def:gwext}
Given
\[
\phi_1, \ldots, \phi_n \in \mathcal{H}
\]
and nonnegative integers $a_1, \ldots, a_n$, the associated {\it degree-$\beta$, genus-$g$, $\T$-equivariant GW correlator} is defined by
\[
\langle \phi_1\psi^{a_1} \cdots \phi_n \psi^{a_n}\rangle^{X,\T}_{g,n,\beta}:=\int_{[\M_{g,n}(X,\beta)]_\T^{\vir}}\ev_1^*(\phi_1)\psi_1^{a_1} \cdots \ev_n^*(\phi_n)\psi_n^{a_n}.
\]
\end{definition}

\subsection{Givental's symplectic formalism}
\label{sec:givental}

The genus-zero GW invariants can be encoded geometrically as an over-ruled cone in an infinite-dimensional vector space.  Following Givental, we define the symplectic vector space
\[
\V_{\T}^X := \H[z,z^{-1}][[Q^{-\frac{1}{d}}]](\alpha)
\]
with the symplectic form
\[
\Omega_{\T}(f_1, f_2) := \text{Res}_{z=0}\bigg(f_1(-z), f_2(z)\bigg)_{\T},
\]
where 
\[
(-,-)_\T=\left\langle \one\;\phi_1\;\phi_2\right\rangle_{0,3,0}^{X,\T}
\]
 is the equivariant Poincar\'e pairing on $X$. Here, $Q$ is the Novikov variable. We view $\mathcal{V}_{\T}$ as a module over the ground ring
\[
\Lambda_{\nov}^{\T} := \C[[Q^{-\frac{1}{d}}]](\alpha).
\]

There is a decomposition
\[
\V_{\T}^X = \V_{\T}^{X+} \oplus \V_{\T}^{X-}
\]
into Lagrangian subspaces, where
\[
\V_{\T}^{X,+} = \H[z][[Q^{-\frac{1}{d}}]](\alpha),
\]
\[
\V_{\T}^{X,-} = z^{-1}\H[z^{-1}][[Q^{-\frac{1}{d}}]](\alpha).
\]
Via this polarization, $\V_{\T}^X$ can be identified with the cotangent bundle $T^*\V_{\T}^{X,+}$ as a symplectic vector space.


Fix a basis $\{\Phi_\mu\}$ of $\H$ such that $\Phi_0=\one$, and let $\{\Phi^\mu\}$ denote the dual basis under the pairing $(-,-)_\T$. This basis yields Darboux coordinates for $\V_{\T}^X$.  Namely, an arbitrary element of $\V_{\T}^X$ can be expressed as
\[
\sum_{k,\mu} q_k^\mu \Phi_\mu z^k + \sum_{k,\mu} p_{k,\mu}\Phi^\mu (-z)^{-k-1}.
\]

The {\it genus-zero generating function} of GW theory is defined by
\[
\mathcal{F}_\T^X(\bftau) = \sum_{\beta,n} \frac{Q^\beta}{n!}\langle \bftau^n(\psi)\rangle^{X,\T}_{0,n,\beta},
\]
where
\[
\bftau^n(\psi)=\bftau(\psi_1)\cdots\bftau(\psi_n)
\]
and
\[
\bftau(z) = \sum_{k,\mu} \tau_k^{\mu}\Phi_\mu z^k.
\]
The sum is over all $n$ and $\beta$ giving a nonempty moduli space $\M_{0,n}(X,\beta)$. 

We view $\mathcal{F}_\T^X$ as a function of the variables $\{q^\mu_k\}$ by way of the {\it dilaton shift}
\[
q^{\mu}_k = 
\begin{cases} \tau^{\mu}_k - 1 & \text{ if } k=1 \text{ and }\mu = 0\\ \tau^\mu_k & \text{ otherwise}.
\end{cases}
\]
It is a fundamental property of GW theory that $\F_\T^X$ satisfies the following three differential equations:
\begin{enumerate}
\item[(SE)]
\[
\frac{1}{2}\left( q, q \right)_{\T}=-\sum_{k\geq 0}\sum_{\mu}q_{k+1}^{\mu}\frac{\partial \F_\T^X}{\partial q_k^{\mu}};
\]
\item[(DE)]
\[
2\F=\sum_{k\geq 0}\sum_{\mu}q_k^{\mu}\frac{\partial \F_\T^X}{\partial q_k^{\mu}};
\]
\item[(TRR)]
\[
\frac{\partial^3\F_\T^X}{\partial q_{k+1}^{\alpha}\partial q_i^\beta \partial q_j^\gamma}=\sum_{\mu,\nu}\frac{\partial^2\F_\T^X}{\partial q_k^\alpha \partial q_0^\mu}g^{\mu\nu}\frac{\partial^3\F_\T^X}{\partial q_0^\nu\partial q_l^\beta\partial q_m^\gamma}, \hspace{.5cm} \forall \alpha,\beta,\gamma,i,j,k.
\]
\end{enumerate}
In these equations, $g^{\mu\nu}$ is the inverse of the matrix corresponding to the pairing $(-,-)_{\T}$, and $q=q_0=\sum q_0^\mu \Phi_\mu$.

Consider the graph of the differential of $\F_\T^X$:
\[
\hat\cL_{\T}^X := \{(\mathbf{q}, \mathbf{p}) \; | \; \mathbf{p} = d_{\mathbf{q}} \mathcal{F}_\T^X\} \subset \V_{\T}^X[[\tau]].
\]
More concretely, a general point of $\hat\cL_\T^X$ has the form
\begin{equation}\label{eq:formalnbd}
 -\mathbf{1}z + \bftau(z) +\sum_{n,\beta,\mu} \frac{Q^\beta}{n!}\left\langle \bftau^{n}(\psi)\; \frac{\Phi_\mu}{-z-\psi}\right\rangle^{X,\T}_{0,n+1,\beta} \Phi^\mu.
\end{equation}
We view $\hat\cL_{\T}^X$ as a formal subspace centered at $-\one z\in\V_\T^{X}$.

The equations (SE), (DE), and (TRR) together imply that the points of $\hat\cL_{\T}^X$ have some very special properties. These properties  can be described by the following geometric interpretation:

\begin{theorem}[Coates--Givental \cite{CoatesGivental}]\label{thm:givental}
$\hat\cL_{\T}^X$ is a formal germ of a Lagrangian cone $\cL_{\T}^X$ with vertex at the origin such that each tangent space $T$ to the cone is tangent to the cone exactly along $zT$.
\end{theorem}

Theorem \ref{thm:givental} implies that the points of the Lagrangian cone are completely determined by any $\dim_{\C(\alpha)}(\H)$-dimensional transverse slice. In particular, one such slice is given by the {\it $J$-function}:
\begin{equation}\label{eq:gwj}
J_{\T}^X(\tau, -z) = -\mathbf{1}z + \tau +\sum_{n,\beta,\mu} \frac{Q^\beta}{n!}\left\langle \tau^{n}\;\frac{\Phi_\mu}{-z-\psi}\right\rangle^{X,\T}_{0,n+1,\beta} \Phi^\mu,
\end{equation}
where
\[
\tau=\sum_{\mu}\tau^{\mu}\Phi_\mu.
\]
Theorem \ref{thm:givental}, along with the string equation, implies that the derivatives of $J_{\T}^X(\tau,-z)$ span the Lagrangian cone:
\begin{equation}\label{eq:ruling}
\cL_{\T}^X=\left\{\left. z\sum_{\mu}c_{\mu}(z)\frac{\partial }{\partial \tau^{\mu}}J_{\T}^X(\tau,-z) \right|c_{\mu}(z)\in\Lambda_{\nov}^{\T}[z]\right\}.
\end{equation}
In practice, we also allow the coefficients $c_\mu(z)$ to be power series in $z$, as long as they converge in some specified topology.

\subsection{Formal subspaces in the GLSM}
\label{formal}

We can analogously encode the genus-zero GLSM invariants as a formal subspace in Givental's symplectic vector space.  First, define the \emph{GLSM I-function}:
\begin{equation}\label{Ifunction}
I_\T^{X,W}(Q,z):=z\sum_{a\in\frac{1}{d}\Z \atop a>0} Q^{-a}\frac{\prod_{i=1}^M\prod_{0\leq b< a w_i \atop \langle b \rangle = \langle a w_i \rangle}(-bz-e_T(\O_X(w_i)))}{\prod_{j=1}^N\prod_{0< b< a d_j \atop \langle b \rangle = \langle a d_j \rangle}(bz+e_T(\O_X(-d_j)))}\one_{( a )}.
\end{equation}
Here, $\O(w_i)$ and $\O(-d_j)$ are the (canonically $\T$-linearized) orbifold line bundles on \[X=[\C^{M+N}\GIT_{\theta<0} G]\] corresponding to the $G$-equivariant line bundles given by the $i$th and $(M+j)$th factors of $\C^{M+N}$, respectively. The class $\one_{(a)}=\one_{(\langle a\rangle)}$ is the fundamental class of the twisted sector $X_{(\langle a \rangle)} \subset \I X$.

In analogy with \eqref{eq:formalnbd}, we define a formal subspace $\hat\cL_\T^{X,W}$ as the collection of points of the form
\begin{equation}
\label{eq:formalGLSM}
I_\T^{X,W}(Q,-z)+\bt(z) +\sum_{n,\beta,\mu} \frac{Q^\beta}{n!}\left\langle \bt^{n}(\psi)\; \frac{\Phi_\mu}{-z-\psi}\right\rangle^{X,W,\T}_{0,n+1,\beta} \Phi^\mu,
\end{equation}
viewed as a formal subspace centered at $I_\T^{X,W}(Q,-z)$ in $\V_\T^{X}$.
Notice that the power of $Q$ can be positive in the sum. Therefore, we view $\hat\cL_\T^{X,W}$ as a formal subspace over the extended ground ring:
\[
\tilde\Lambda_{\nov}^{\T}:=\Lambda_{\nov}^{\T}[Q^{\frac{1}{d}}].
\]

It is not obvious, a priori, that the formal subspace $\hat\cL_{\T}^{X,W}$ has any geometric properties analogous to those described in Theorem \ref{thm:givental}. One consequence of Theorem \ref{maintheorem}, however, is that $\hat\cL_{\T}^{X,W}$ is, in fact, a formal germ of the over-ruled Lagrangian cone $\cL_{\T}^X$.

\section{Localization in the GLSM}
\label{localization}

In this section, we describe the virtual localization formula for the extended genus-zero GLSM invariants. References for virtual localization are Graber--Pandharipande \cite{GP} and Liu \cite{Liu}.

\subsection{Localization graphs}

Virtual localization expresses the genus-zero GLSM correlators
\begin{equation}\label{eq:correlator}
\left\langle \phi_1\psi^{a_1}\cdots\phi_n\psi^{a_n}\right\rangle^{X,W,\T}_{0,n,\beta}
\end{equation}
as a sum of contributions from the $\T$-fixed loci on $\QM_{0,n}^W(X,\beta)$. The fixed loci of $\QM_{0,n}^W(X,\beta)$ are indexed by decorated trees $\Gamma$. For a tree $\Gamma$, let $V(\Gamma), E(\Gamma),$ and $F(\Gamma)$ denote the sets of vertices, edges, and flags of $\Gamma$, respectively. The localization trees are decorated as follows:
\begin{itemize}
	\item Each vertex $v$ is decorated by an index $k_v\in\{1,\dots,N\}$ and a degree $\beta_v\in\frac{1}{d}\Z$.
	\item Each edge $e$ is decorated by a degree $\beta_e\in\frac{1}{d}\Z$.
	\item Each flag $(v,e)$ is decorated by a multiplicity $m_{(v,e)}\in\frac{1}{d}\Z/\Z$.
\end{itemize}
In addition, $\Gamma$ is equipped with a map
\[s: \{1, \ldots, n\} \rightarrow V(\Gamma)\]
assigning marked points to the vertices.  Let $E_v$ be the set of edges adjacent to $v$, and set
\[
\val(v):=|E_v|+|s^{-1}(v)|.
\]

A point in the fixed locus $F^W_\Gamma$ indexed by the decorated graph $\Gamma$ can be described as follows:
\begin{itemize}
\item Each vertex $v\in V(\Gamma)$ corresponds to a maximal connected component $C_v$ over which $p_j=0$ for $j\neq k_v$, and $\deg(L|_{C_v}) = \beta_v$.
\item Each edge $e\in E(\Gamma)$ with adjacent vertices $v$ and $v'$ corresponds to an orbifold line $C_e$ with two marked points $q_v,q_{v'}$, over which $p_j=0$ for $j\neq k_v,k_{v'}$.  The section $p_{k_v}$ vanishes only at $q_{v'}$, while the section $p_{k_{v'}}$ vanishes only at $q_v$, and $\deg(L|_{C_e})= \beta_e$. 
\item The set $s^{-1}(v)\subset\{1,\dots,n\}$ indexes the marked points supported on $C_v$.
\item If either $\val(v)>2$, or $\val(v)=2$ and $\beta_v<0$, then the flag $(v,e)\in F(\Gamma)$ corresponds to a node attaching $C_v$ to $C_e$, and $m_{(v,e)}$ gives the multiplicity of $L$ on the vertex branch of the node. If $\val(v)=2$ and $\beta_v=0$, then $C_v$ is the smooth point $q_v\in C_e$ and $-m_{(v,e)}$ is the multiplicity of $L$ at $q_v$.
\end{itemize}
We denote by $m_i \in \Q/\Z$ the twisted sector of $H^*_{CR}(X)$ from which the insertion $\phi_i$ is drawn.

The non-emptiness of $F_{\Gamma}^W$ imposes a number of constraints on the decorations.  For example, for each edge $e$ with adjacent vertices $v$ and $v'$, one must have
\[\beta_e + m_{(v,e)} + m_{(v',e)} \in \Z\]
and $\beta_e < 0$.

The contribution of a graph $\Gamma$ to the localization expression for \eqref{eq:correlator} can be subdivided into vertex, edge, and flag contributions, which we outline below.

\subsection{Vertex contributions}

We call a vertex \emph{stable} if either $\val(v)>2$ or $\val(v)=2$ and $\beta_v<0$; otherwise, we say the vertex is \emph{unstable}.

Let $v$ be a stable vertex, and let $F_v^W$ denote the $\T$-fixed locus in $\QM_{0,\val(v)}^W(X,\beta_v)$ for which $p_j=0$ for $j\neq k_v$. Let $N_F^W$ denote the virtual normal bundle of $F_v^W$ in $\QM_{0,\val(v)}^W(X,\beta_v)$. We define the vertex contribution of the stable vertex $v$ to be
\[
\Contr^W_\Gamma(v):=\int_{[F_v^W]_T^{\vir}}\frac{\prod_{i\in s^{-1}(v)} \ev_i^*(\one_{(m_i)}^{k_v})\psi^{a_i}}{e_T(N_F^W)}\prod_{e\in E_v}\frac{d_{k_v}\ev_e^*(\one_{(m_{(v,e)})}^{k_v})}{\frac{\alpha_{k_{v'}}-\alpha_{k_{v}}}{\beta_e}-\psi_e}.
\]

The stability condition in the definition of stable LG quasi-maps ensures that the only unstable vertices are those for which $s^{-1}(v)=\{i_v\}$, $E_v=\{e\}$, and $\beta_v=0$. Let $v_e$ be the other vertex adjacent to $e$. We define the vertex contribution of the unstable vertex $v$ to be
\[
\Contr^W_\Gamma(v):=\left(\frac{\alpha_{k_{v}}-\alpha_{k_{v_e}}}{\beta_e}\right)^{a_{i_v}}.
\]
Note that the expression in the parentheses is nothing more than $\psi_{i_v}$ restricted to the fixed locus $F^W_\Gamma$.

\subsection{Edge contributions}

Let $e$ be an edge with adjacent vertices $v$ and $v'$, and let $F_e$ denote the $\T$-fixed locus in $\QM_{0,(-m_{(v,e)},-m_{(v',e)})}^W(X,\beta_e)$ for which $p_j=0$ for $j\neq k_v,k_{v'}$, the section $p_{k_v}$ vanishes only at $q_{v'}$, and the section $p_{k_{v'}}$ vanishes only at $q_{v}$. Let $N_e$ denote the virtual normal bundle of $F_e$. We define
\begin{align}\label{edge}
\nonumber\Contr^W_\Gamma(e)&:=\int_{[F_e]^{\vir}}\frac{1}{e_T(N_e)}\\
&=\frac{1}{d_{k_v}d_{k_{v'}}\beta_e} \cdot \frac{e_\T\left(\bigoplus_{i=1}^M R^1\pi_*\cL^{\otimes w_i}\right)}{e_\T\left(\left(\bigoplus_{j=1}^N R^0\pi_*\cL^{\otimes -d_j}\right)^{mov}\right)},
\end{align}
where the superscript $mov$ denotes the moving part with respect to the $\T$-action. To arrive at \eqref{edge}, we are using the fact that $\omega_{\log}\cong\O$ on $C_e$.

The expression \eqref{edge} can be made explicit, following \cite{Liu}:
\[
\Contr^W_\Gamma(e)=\frac{1}{d_{k_v}d_{k_{v'}}\beta_e}\frac{\prod_{i}\prod_{0<b<-\beta_e w_i \atop \langle b \rangle =\langle m_{k_{v}}w_i \rangle}\left(\frac{b}{\beta_e}(\alpha_{k_{v'}}-\alpha_{k_{v}})-w_i\alpha_{k_{v}}\right)}{\prod_{j}\prod_{0\leq b\leq-\beta_e d_j \atop \langle b \rangle =\langle m_{k_{v}}d_j \rangle}'\left(\frac{b}{\beta_e}(\alpha_{k_{v}}-\alpha_{k_{v'}})+d_j(\alpha_{k_{v}}-\alpha_j)\right)},
\]
where $\prod'$ in the denominator denotes the product over all nonzero factors.

\subsection{Flag contributions}\label{flag}

Let $(v,e)$ be a flag at a stable vertex $v$. Let $N_{k_v}^{m_{(v,e)}}$ denote the normal bundle of the unique $\T$-fixed point in $X_{(m_{(v,e)})}$ where $p_{k_v}\neq 0$. We define the flag contribution of the stable flag $(v,e)$ to be
\[
\Contr^W_\Gamma(v,e):=e_T\left(N_{k_v}^{m_{(v,e)}}\right).
\]
If $v$ is an unstable vertex, then we define $\Contr^W_\Gamma(v,e)=1$.

\subsection{Total graph contributions}

Combining all of these contributions, the genus-zero GLSM correlators are given by:
\[
\left\langle \phi_1\psi^{a_1}\cdots\phi_n\psi^{a_n}\right\rangle^{X,W,\T}_{0,n,\beta}=\sum_{\Gamma}\Contr^W_\Gamma,
\]
where $\Contr^W_\Gamma$ equals
\[
\frac{1}{|\Aut(\Gamma)|}\prod_{v\in V(\Gamma)}\Contr^W_\Gamma(v)\prod_{e\in E(\Gamma)}\Contr^W_\Gamma(e)\prod_{(v,e)\in F(\Gamma)}\Contr^W_\Gamma(v,e).
\]

\subsection{Comparison with GW theory}

The virtual localization formula for the GW theory of $X$ is developed carefully in \cite{Liu}. For the reader's convenience, we briefly compare it to the above localization formula in the GLSM.

The first difference between the localization formulas is that the GW localization graphs are a superset of the GLSM localization graphs. Indeed, there are localization graphs in GW theory with vertices of valence one, which correspond to rational tails of the source curve. The contributions from these rational tails, however, do not play a role in our proofs, so we do not discuss them further.  Whenever an edge or flag has no adjacent vertex of valence one, its contribution to the GW localization formula is exactly the same as in the GLSM.

The second major difference occurs at stable vertices, and it does play a role in what follows.  Namely, in GW theory, $\Contr^W_\Gamma(v)$ is replaced by
\[
\Contr_\Gamma(v):=\int_{[F_v]_T^{\vir}}\frac{\prod_{i\in s^{-1}(v)} \ev_i^*(\one_{(m_i)}^{k_v})\psi^{a_i}}{e_T(N_F)}\prod_{e\in E_v}\frac{d_{k_v}\ev_e^*(\one_{(m_{(v,e)})}^{k_v})}{\frac{\alpha_{k_{v'}}-\alpha_{k_v}}{\beta_e}-\psi_e};
\]
that is, the integral over $F_v^W \subset \QM_{0,\val(v)}^W(X,\beta_v)$ becomes an integral over $F_v \subset \M_{0,\val(v)}(X,\beta_v)$, which is similarly defined as the $\T$-fixed locus for which $p_j=0$ for $j\neq k_v$.

\section{Vertex Correspondence}
\label{hybrid}

In this section, we make a comparison of the GW and GLSM invariants appearing in the stable vertex terms of the localization graphs of the previous section. 

\subsection{Twisted GW theory}
\label{twistedGW}

Recall that $P_k$ denotes the unique $\T$-fixed point of $X$ for which $p_k \neq 0$.  After unraveling the definitions, we see that the stable vertex contributions $\Contr_\Gamma(v)$ for GW theory encode invariants of the form
\begin{equation}\label{eq:twistedgw}
\int_{[\M_{0,n}(P_k,\beta)]}\ev_1^*(\phi_1) \psi^{a_1} \cdots \ev_n^*(\phi_n) \psi^{a_n}\cdot e_\T^{-1}\left(R\pi_*\cT_k\right),
\end{equation}
where $\phi_1, \ldots, \phi_n \in H_{CR,\T}^*(P_k)$ and 
\[
\cT_k:=\bigoplus_{i=1}^M \cL^{\otimes w_i}\oplus\bigoplus_{j\neq k} \cL^{\otimes -d_j}
\]
with $\T$-weights $w_i\alpha_k$ for each $i\in\{1,\dots,M\}$ and $d_j\alpha_j-d_j\alpha_k$ for each $j\in\{1,\dots,N\}\setminus\{k\}$.

The expression \eqref{eq:twistedgw} is an example of a genus-zero {\it twisted correlator} on $\M_{0,n}(P_k,\beta)$. Coates--Givental \cite{CoatesGivental} have developed a general framework for working with such twisted theories, which we now recall. 

For any choice of parameters
\[s=\left\{s_l^i,\ts_l^j{} \; \left| \; \substack{i \in \{1,\ldots, M\},\\j\in \{1, \ldots, N\} \setminus \{k\},\\ l\geq 0}\right.\right\},\]
a characteristic class on vector bundles of the form \[T=\bigoplus_{i=1}^M U_i\oplus\bigoplus_{j\neq k}V_{k}\]
can be defined by 
\begin{equation}\label{eq:charclass}
c(T):=\prod_{i=1}^M\exp\left(\sum_{l \geq 0} s_l^i\ch_l(U_i)\right)\cdot \prod_{j\neq k}\exp\left(\sum_{l \geq 0} \ts_l^{j}\ch_l(V_j)\right).
\end{equation}
This class is multiplicative, and can thus be extended to $K$-theory.

Given such parameters, we define \emph{c-twisted correlators} in all genera by
\[
\langle \phi_1 \psi^{a_1} \cdots \phi_n \psi^{a_n}\rangle^{P_k,c}_{g,n,\beta}:=\int_{[\M_{g,n}(P_k,\beta)]}\ev_1^*(\phi_1) \psi^{a_1} \cdots \ev_n^*(\phi_n) \psi^{a_n}c\left(R\pi_*\cT_k\right).
\]
When $s^i_l,\ts_l^{j}=0$ for all $i,j,l$, we obtain the \emph{untwisted correlators}, which we denote by $\langle \phi_1 \psi^{a_1} \cdots \phi_n \psi^{a_n}\rangle^{P_k,\un}_{g,n,\beta}$.

\begin{remark}\label{rmk:gwun}
The specific choice
\begin{align*}
s_l^i = \begin{cases} -\displaystyle \ln\left(w_i\alpha_k\right) & \text{ if } l = 0\\ \; & \; \\ \displaystyle\frac{(l-1)!}{(-w_i\alpha_k)^l} & \text{ if } l > 0\end{cases}
&\hspace{1cm}\ts_l^{j} = \begin{cases} -\displaystyle \ln(d_j\alpha_j-d_j\alpha_k) & \text{ if } l = 0\\ \; & \; \\ \displaystyle\frac{(l-1)!}{(d_j\alpha_k-d_j\alpha_j)^l} & \text{ if } l > 0\end{cases}
\end{align*}
yields a characteristic class $c_{k}(-)$ such that
\[
c_{k}\left(R\pi_*\cT_k\right)=e_\T^{-1}\left(R\pi_*\cT_k\right).
\]
\end{remark}

For any choice of multiplicative characteristic class $c$, the genus-zero $c$-twisted GW correlators define an axiomatic Gromov--Witten theory on the symplectic vector space
\[
\V^{P_k}_c= H_{CR}^*(P_k)((Q^{-\frac{1}{d_k}},z))[[s]]
\]
with symplectic form induced by the twisted pairing
\begin{equation}\label{eq:gwtw}
(\phi_1, \phi_2)^{P_k}_c = \langle \one^k \; \phi_1 \; \phi_2 \rangle^{P_k, c}_{0,3,0}.
\end{equation}
We denote the corresponding Lagrangian cone by $\cL^{P_k}_c\subset\V^{P_k}_c$.  Notice that in the twisted theory, we allow Laurent series in $z$ rather than only Laurent polynomials; this will be important later.

The following theorem relates $c$-twisted GW correlators to their untwisted versions:

\begin{theorem}[Tseng \cite{Tseng}]\label{thm:tseng}
Define the symplectic transformation
\[\Delta: \mathcal{V}^{P_k}_{\un} \rightarrow \mathcal{V}^{P_k}_{c}\]
by
\[\Delta = \hspace{-0.2cm}\bigoplus_{m \in \frac{1}{d_k}\Z/\Z}\hspace{-0.2cm}\exp\bigg(\sum_{l \geq 0} \bigg[\sum_{i=1}^M s_l^i\frac{B_{l+1}\left(\left\langle w_im\right\rangle\right)}{(l+1)!}+\sum_{j\neq k}\ts_l^{j}\frac{B_{l+1}\left(\left\langle -d_jm\right\rangle\right)}{(l+1)!}\bigg]z^{l}\bigg).\]
Then
\[\Delta(\cL^{P_k}_{\un}) = \cL^{P_k}_c.\]
\end{theorem}

More explicitly, $\Delta$ acts diagonally with respect to the decomposition of $H^*_{CR}(P_k)$ into twisted sectors, and the $m$th component gives its action on the sector indexed by $m$.

\subsection{Twisted theory on weighted spin curves}

A similar formalism applies to the stable vertex terms in the GLSM.  First, however, some further comments on the underyling moduli space are in order.

At a stable vertex $v$ in a GLSM localization graph $\Gamma$ with $k_v = k$, the moduli space $F_v$ parameterizes a prestable marked orbifold curve $C_v$ along with an orbifold line bundle $L$ and an isomorphism
\[
L^{\otimes d_k}\cong\omega_{\log}\otimes\O(-B),
\]
where $B$ is the base locus of $\sigma|_{C_v}$. The stability condition is equivalent to insisting that $\omega_{\log}\otimes\O(\epsilon B)$ is ample for all $\epsilon>0$. These moduli spaces where introduced by the second author and Ruan in \cite{RR} under the name \emph{$\epsilon=0$ weighted spin curves}.

More generally, let $\M_{g,n}^d(\beta)^\epsilon$ denote the moduli space parametrizing genus-$g$ prestable marked orbifold curves $(C;q_1,\dots,q_n)$ along with an effective divisor $B$, a degree-$\beta$ line bundle $L$, and an isomorphism
\[
L^{\otimes d}\cong \omega_{\log}\otimes \O(-B),
\]
satisfying:
\begin{enumerate}
	\item The support of $B$ is disjoint from the marks and nodes of $C$.
	\item For every $q \in C$ with isotropy group $G_q$, the homomorphism $G_q \rightarrow \C^*$ giving the action of the isotropy group on the bundle $\bigoplus_i L^{\otimes w_i} \oplus L^{\otimes -d_k}$ is injective.
	\item $\omega_{\log} \otimes \O(\delta B)$ is ample for all $\delta > \epsilon$.
\end{enumerate}
Let $\M_{g,\vec m}^{d}(\beta)^{\epsilon}$ denote the component of the moduli space on which $L$ has multiplicity $m_i$ at the $i$th marked point.

Then $\Contr_\Gamma(v)$ encodes invariants of the form
\begin{equation}\label{eq:twistedspin}
\int_{[\M_{0,\vec m}^{d_k}(\beta)^{\epsilon=0}]}\psi_1^{a_1} \cdots \psi_n^{a_n}\; e_\T^{-1}\left(R\pi_*\cT_k^{W}\right),
\end{equation}
where 
\[
\cT_k^{W}:=\bigoplus_{i=1}^M \cL^{\otimes w_i}\oplus\bigoplus_{j\neq k} \cL^{\otimes -d_j}\otimes\omega_{\log}
\]
with $\T$-weights $w_i\alpha_k$ for each $i\in\{1,\dots,M\}$ and $d_j\alpha_j-d_j\alpha_k$ for each $j\in\{1,\dots,N\}\setminus\{k\}$.

In analogy with twisted GW theory, we make the following definition.

\begin{definition} For a characteristic class $c$ defined as in \eqref{eq:charclass} and a choice of $\epsilon\geq 0$, we define the {\it c-twisted spin correlators} by
\[
\langle \one_{(m_1)}^{k} \psi^{a_1} \cdots \one_{(m_n)}^k \psi^{a_n}\rangle^{P_k,W,\epsilon,c}_{g,n,\beta}:=\hspace{-.1cm}\int_{[\M_{0,\vec m}^{d_k}(\beta)^{\epsilon}]} \hspace{-.2cm}\psi_1^{a_1} \cdots \psi_n^{a_n}\;c(R\pi_*\mathcal{T}_k^{W}).
\]
\end{definition}

\begin{remark}\label{rmk:glsmun}
For $\epsilon=0$, the specific choice made in Remark \ref{rmk:gwun} yields the vertex contributions from the GLSM localization formula. 
\end{remark}

For $\epsilon\gg 0$ (denoted $\epsilon=\infty$) and any choice of characteristic class $c$, the genus-zero twisted spin correlators define an axiomatic Gromov--Witten theory on the symplectic vector space $\V^{P_k}_c$, which is a twisted version of FJRW theory. More specifically, define the genus-zero potential by
\[
\F_{k,c}^{W}(\bt)=\sum_{n,\beta}\frac{Q^{\beta}}{n!}\left\langle\bt(\psi)^n\right\rangle_{0,n,\beta}^{P_k,W,\infty,c},
\]
where 
\[
\bt(\psi)=\sum_{m,l} t_l^m \one_{(m)}^k\psi^l.
\]
It is a fundamental property of FJRW theory that, after the dilaton shift
\[
q^{m}_l = 
\begin{cases} t^{m}_l - Q^{\frac{1}{d_k}} & \text{ if } l=1 \text{ and }m = 1/d_k\\ t^m_l & \text{ otherwise},
\end{cases}
\]
the potential $\F_{k,c}^{W}$ satisfies the equations (SE), (DE), and (TRR) described in Section \ref{sec:givental}.

\begin{remark}
In the theory of spin curves, there exists a forgetful map only on the component where the last marked point has multiplicity $\frac{1}{d_k}$; thus, $\one^k_{(1/d_k)}$ plays the role of the unit in this theory. Moreover, the forgetful map changes the degree of $L$. This explains why the dilaton shift differs from that in GW theory. In addition, it is straightforward to check that the twisted pairing \eqref{eq:gwtw} is recovered by:
\begin{equation*}
(\phi_1, \phi_2)^{P_k}_{c} = \langle  \one_{(1/d_k)}^k\;\phi_1 \; \phi_2  \rangle^{P_k,W,\infty, c}_{0,3,1/d_k}.
\end{equation*}
\end{remark}

We denote the Lagrangian cone associated to the $\epsilon=\infty$ $c$-twisted spin theory by $\cL^{P_k,W}_c\subset\V^{P_k}_c$. 
Two results about this cone will be important in what follows; briefly, these are:
\begin{enumerate}
\item Wall-crossing:  The $c$-twisted spin correlators for any $\epsilon$ can be recovered from the $\epsilon=\infty$ $c$-twisted spin correlators, by relating them to $\cL^{P_k,W}_c$.
\item Symplectomorphisms: The $\epsilon=\infty$ $c$-twisted spin correlators for any $c$ can be recovered from the $\epsilon=\infty$ untwisted correlators, by giving a symplectomorphism taking $\cL^{P_k,W}_{un}$ to $\cL^{P_k,W}_c$.
\end{enumerate}
Below, we make these two facts explicit.

\subsection{Wall-crossing}

Fix a characteristic class $c$ as above.  Although the twisted spin correlators can, a priori, only be encoded in an over-ruled Lagrangian cone for $\epsilon=\infty$, we can still define a formal subspace for any $\epsilon$, analogously to Section \ref{formal}.  To do so, we must define an $\epsilon$-dependent $I$-function and this can be done using graph spaces, following \cite{Bertram}, \cite{CFK2}, and \cite{RR}.

More specifically, let $G\M^{d_k}_{0, n+1}(\beta)^{\epsilon}$ be the graph space of weighted spin curves with the additional data of a parameterization of one component of the source curve $C$. Stability, here, only requires $\omega_{\log} \otimes \O(\delta B)$ to be ample (for all $\delta > \epsilon$) on the non-parameterized components. Let $G\M^{d_k}_{0, \vec m + m}(\beta)^{\epsilon}$ be the component of the graph space where the multiplicity of $L$ on the $i$th marked point is $m_i$ for $i\leq n$ and the multiplicity of $L$ on the last marked point is $m$. We define twisted correlators on this moduli space by integration:
\begin{equation}\label{graphcorrelator}
\int_{[G\M^{d_k}_{0, \vec m+m}(\beta)^{\epsilon}]} \psi_1^{a_1} \cdots \psi_n^{a_n}\;c(R\pi_*\mathcal{T}_k^{W}).
\end{equation}

The graph space admits a $\C^*$ action induced by scaling the coarse coordinates of the parameterized component:
\[
t\cdot[y_0, y_1]:=[ty_0,y_1].
\]
Let $z$ denote the equivariant parameter of this action.

There is a special $\C^*$-fixed locus in $G\M^{d_k}_{0, \vec m+m}(\beta)^{\epsilon}$ where the last marked point is $[0:1]$ and the rest of the marked points and basepoints lie over $[1:0]$. Denote this fixed locus by $F_{\vec m+m,\beta}^\epsilon$. Let $\mathrm{Res}\left(F_{\vec m+m,\beta}^\epsilon\right)$ denote the equivariant residue of $F_{\vec m+m,\beta}^\epsilon$ with respect to the integral \eqref{graphcorrelator}. Whenever $\M^{d_k}_{0, \vec m+m}(\beta)^{\epsilon}$ is nonempty (that is, when $n>1$, when $n=1$ and $\beta<0$, or when $n=0$ and $\beta\leq -\frac{1+\epsilon}{d_k\epsilon}$), a standard computation shows that 
\[
\mathrm{Res}\left(F_{\vec m+m,\beta}^\epsilon\right)=\frac{-1}{z^2}\left\langle \one_{(m_1)}^{k} \psi^{a_1} \cdots \one_{(m_n)}^k \psi^{a_n}\;\frac{\one_{(m)}^k}{z-\psi}\right\rangle^{P_k,W,\epsilon,c}_{0,n+1,\beta}.
\]
When $n=1$ and $\beta=0$, we have
\[
\mathrm{Res}\left(F_{m_1+m,\beta=0}^\epsilon\right)=\begin{cases}
-\frac{c\left(N_{k}^{m}\right)}{d_kz^2}(-z)^{a_1} & m=-m_1\\
0 & \text{otherwise.}
\end{cases}
\]
When $n=0$ and $\beta> -\frac{1+\epsilon}{d_k\epsilon}$, we have
\[
\mathrm{Res}\left(F_{m,\beta}^\epsilon\right)=\begin{cases}
\frac{-c_{\C^*}\left(R\pi_*\cT_k^W \right)}{d_kz^{a+1}a!} & \exists a 
\in \N \text{ s.t. }\beta=-\frac{a+1}{d_k}\text{, }m=\left\langle -\frac{a+1}{d_k}\right\rangle\\
0 &\text{otherwise.}
\end{cases}
\]

Packaging these residues in a generating series, define $\hat\cL_c^{P_k,W,\epsilon}$ as the formal subspace of $\V_c^{P_k}[[t]]$ consisting of points of the form
\[
I_c^{P_k,W,\epsilon}(Q,-z)+\bt(z)+\sum_{n,\beta \atop m\in\frac{1}{d_k}\Z/\Z}\frac{Q^{\beta}}{n!}\left\langle\bt(\psi)^n\;\frac{\one_{(m)}^k}{-z-\psi}\right\rangle_{0,n+1,\beta}^{P_k,W,\epsilon,c}(\one_{(m)}^k)^\vee,
\]
where $(-)^\vee$ denotes the dual with respect to the twisted pairing $(-,-)_c^{P_k}$ and
\[
I_c^{P_k,W,\epsilon}(Q,z):=\frac{z}{d_k}\sum_{a\in\N \atop 0\leq a <\epsilon}\frac{Q^{-\frac{a+1}{d_k}}}{z^aa!}\;c_{\C^*}\left(R\pi_*\cT_k^W \right)\left(\one^k_{\left(-\frac{a+1}{d_k}\right)}\right)^\vee
.\]

\begin{theorem}[c.f. Ross--Ruan \cite{RR}]\label{thm:wallcrossing}
For any $\epsilon\geq 0$, $\hat\cL_c^{P_k,W,\epsilon}$ is a formal germ of the over-ruled Lagrangian cone $\cL_c^{P_k,W}$.
\end{theorem}

\begin{proof}
This result is a twisted version of the main result in \cite{RR}. Using localization on graph spaces, the techniques of \cite{RR} can be applied directly to prove this result.
\end{proof}

\begin{remark}\label{twistedI}
When $\epsilon = 0$ and the characteristic class $c$ is chosen as in Remark \ref{rmk:gwun}, the twisted $I$-function can be computed explicitly:
\[
I_{c_k}^{P_k,W}(Q,z)=z\sum_{a\in\frac{1}{d_k}\Z \atop a>0}Q^{-a}\frac{\prod_i\prod_{0\leq b <aw_i \atop \langle b \rangle =\langle aw_i \rangle }(-bz-w_i\alpha_k)}{\prod_{j}\prod_{0<b<ad_j \atop \langle b \rangle = \langle ad_j \rangle}(bz+d_j(\alpha_k-\alpha_j))}\one_{\left( a\right)}^k.
\]
In particular, the localization isomorphism immediately implies that
\[
I_\T^{X,W}(Q,z)=\sum_{k=1}^N I_{c_k}^{P_k,W}(Q,z),
\]
where $I_\T^{X,W}$ is defined in \eqref{Ifunction}.
\end{remark}

\subsection{Symplectomorphisms}

The relationship between the twisted and untwisted cones for $\epsilon = \infty$ is given by a result precisely analogous to Tseng's:

\begin{theorem}\label{thm:clader}
Let $\Delta$ be the symplectic transformation defined in Theorem \ref{thm:tseng}. Then 
\[
\Delta(\cL^{P_k,W}_{\un}) = \cL^{P_k,W}_c.
\]
\begin{proof}
This is simply a twisted version of the symplectomorphism computed by Lee--Priddis--Shoemaker in Theorem 4.3 of \cite{LPS}, and the proof is a straightforward generalization of theirs.  In particular, the key point is that the action of the quantized operator $\widehat{\Delta}$ on total descendant potentials is defined in terms of the dilaton shift, which differs in the twisted GW and twisted spin cases.  This difference precisely accounts for the discrepancy in the characteristic classes of $\mathcal{T}_k$ and $\mathcal{T}_k^W$.
\end{proof}
\end{theorem}

\subsection{Comparison of untwisted theories}

When $\epsilon = \infty$ and $c = 1$, the (untwisted) spin theory can be directly related to untwisted GW theory.

\begin{lemma}\label{lem:untwisted}
Suppose that $\phi_k = \one^k_{(m_k)}$ for $k=1, \ldots, n$.  Then
 \[
\langle \phi_1 \psi^{a_1} \cdots \phi_n \psi^{a_n}\rangle^{P_k,\un}_{g,n,\beta=-\sum m_k}=\langle \phi_1 \psi^{a_1} \cdots \phi_n \psi^{a_n}\rangle^{P_k,W,\infty,\un}_{g,n,\beta=\frac{2g-2+n}{d_k}-\sum m_k}.
\]
In particular, both correlators are equal to
\[
d_k^{2g-1}\int_{[\M_{g,n}]} \psi_1^{a_1} \cdots \psi_n^{a_n}. 
\]
\end{lemma}

\begin{proof}
The condition on the degrees ensures that the moduli spaces $\M_{g,\vec m}(P_k,\beta)$ and $\M_{g,\vec m}^{d_k}(\beta)^\infty$ are nonempty. The explicit formula follows from the fact that both moduli spaces admit degree-$d_k^{2g-1}$ maps to $\M_{g,n}$ and the $\psi$-classes are pulled back via this maps.
\end{proof}

\begin{lemma}\label{thm:uncones}
We have an identification of untwisted Lagrangian cones:
\[
\cL^{P_k}_{\un}=\cL^{P_k,W}_{\un}.
\]
\end{lemma}

\begin{proof}
In either theory, there is a $J$-function, defined by
\[J^{P_k}_{\un}(\tau, z) = \one^k_{(0)}z + \tau + \sum_{n,\beta,m} \frac{Q^{\beta}}{n!} \left\langle \tau^n \frac{\one^k_{(m)}}{z-\psi} \right\rangle^{P_k,\un}_{0,n+1,\beta}\left(\one^k_{(m)}\right)^{\vee}\]
in GW theory, and
\[J^{P_k,W}_{\un}(t, z) = Q^{\frac{1}{d_k}}\one^k_{(1/d_k)}z + t + \sum_{n,\beta,m} \frac{Q^{\beta}}{n!} \left\langle t^n \frac{\one^k_{(m)}}{z-\psi} \right\rangle^{P_k,W,\un}_{0,n+1,\beta}\left(\one^k_{(m)}\right)^{\vee},\]
in the GLSM.  Here,
\[\tau = \sum_m \tau^m \one^k_{(m)},\]
where the sum runs over a basis for $H^*_{CR}(P_k)$, and $t$ is defined similarly.

The Lagrangian cone for each theory is spanned by linear combinations of derivatives of the corresponding $J$-function.  Thus, it suffices to prove that there exists a change of variables $\tau = \tau(t)$ such that
\[
J^{P_k}_{\un}(\tau,z)=z\sum_{m} c_{m}(z) \frac{\partial}{\partial t^{m}}J^{P_k,W}_{\un}(t,z).
\]
By matching the linear coefficients of $z$, this is equivalent to proving that
\[
J^{P_k}_{un}(\tau,z)=z\frac{\partial}{\partial t^{0}}J^{P_k,W}_{un}(t,z).
\]
Define $\tau = \tau(t)$ by
\[
\tau^{m}:=Q^{\frac{1}{d_k}}t^{m}.
\]
Then:
\begin{align*}
z\frac{\partial}{\partial t^{0}}J^{P_k,W}_{\un}(t,z)&=z\one^k_{(0)}+z\sum_{\beta,n,m}\frac{Q^{\beta}}{n!}\left\langle\one^k_{(0)}\; t^n\;\frac{\one_{(m)}^k}{z-\psi}\ \right\rangle_{0,n+2,\beta}^{P_k,W,\un}\left(\one_{(m)}^k\right)^\vee\\
&=z\one^k_{(0)}+z\sum_{\beta,n,m}\frac{Q^{\beta-\frac{n}{d_k}}}{n!}\left\langle\one^k_{(0)}\; \tau^n\;\frac{\one_{(m)}^k}{z-\psi}\ \right\rangle_{0,n+2,\beta-\frac{n}{d_k}}^{P_k,\un}\left(\one_{(m)}^k\right)^\vee\\
&=J^{P_k}_{\un}(\tau,z).
\end{align*}
The second equality follows from Lemma \ref{lem:untwisted} and the third equality follows from the string equation in GW theory.
\end{proof}

\subsection{Comparison of twisted theories}

\begin{corollary}\label{cor:twisted}\label{cor:twistedid}
We have an identification of $c$-twisted Lagrangian cones:
\[
\cL^{P_k}_{c}=\cL^{P_k,W}_{c}.
\]
\end{corollary}

\begin{proof}
This follows from the fact that the untwisted cones are identified (Lemma \ref{thm:uncones}) along with the fact that the symplectomorphism taking the untwisted to the twisted cone is the same in either case (Theorems \ref{thm:tseng} and \ref{thm:clader}).
\end{proof}

In conjunction with Theorem \ref{thm:wallcrossing}, this completes the comparison of the GW and GLSM correlators appearing at the vertices of the localization graphs.

\section{Formal subspace characterization}
\label{proof}

Having identified the correlators appearing as vertex contributions in the GLSM and GW theory, we must leverage this comparison to relate the full localization expressions for the two theories.  To do so, we will need a characterization of points on the formal subspace $\hat\cL_{\T}^{X,W}$. This characterization is motivated by the cone characterization appearing in work of Coates--Corti--Iritani--Tseng (\cite{CCIT2}, Theorem 41).

First, we provide a natural notion of what it means to be a point on the cone over an auxiliary set of formal parameters $x=(x_1,\dots,x_K)$.

\begin{definition}
	A \emph{$\tilde\Lambda_{\nov}^{\T}[[x]]$-valued point of $\hat\cL_{\T}^{X,W}$} is a point of $\V_{\T}^X[[x]]$ of the form
	\[
	I_\T^{X,W}(Q,-z) + \bt(z) +\sum_{n,\beta,\mu} \frac{Q^\beta}{n!}\left\langle \bt(\psi)^n\; \frac{\Phi_\mu}{-z-\psi}\right\rangle^{X,W,\T}_{0,n+1,\beta} \Phi^\mu
	\]
	for some $\bt(z)\in\V_{\T}^{X,+}[[x]]$ for which $\bt(z)|_{x=0}=0.$
\end{definition}

Notice that $\bt(z)$ is defined over the base ring and, in particular, depends formally on the Novikov parameter $Q$.

The analogous definition can be made at each fixed point.  Namely, we define a $\tilde\Lambda_{\nov}^{\T}[[x]]$-valued point of $\hat\cL_{c_k}^{P_k,W}$ to be a point of $\V_{c_k}^{P_k}[[x]]$ of the form
\[
I_{c_k}^{P_k,W}(Q,-z) + \bt_k(z) +\sum_{n,\beta,m} \frac{Q^\beta}{n!}\left\langle \bt_k(\psi)^n\; \frac{\one_{(m)}^k}{-z-\psi}\right\rangle^{P_k,W,\epsilon=0,c_k}_{0,n+1,\beta} \left(\one_{(m)}^k\right)^\vee
\]
for some $\bt_k(z)\in\V_{c_k}^{P_k,+}[[x]]$ for which $\bt_k(z)|_{x=0}=0.$

We now fix the notation required in the statement of the characterization.  For any $\bff=\bff(z)\in\V_\T^X$, let $\bff_k$ denote the restriction of $\bff$ to $H_{CR}^*(P_k)$ and let $\bff_{k,m}$ denote the coefficient of the fixed-point basis element $\one_{(m)}^k$. 

For a given $k\neq k'$, $m\in\frac{1}{d_k}\Z/\Z$, and $m'\in\frac{1}{d_{k'}}\Z/\Z$, set
\[
E_{k,k'}^{m,m'}:=\{\beta\in\Z-m-m'\;|\; \beta<0\}.
\]
That is, $E_{k,k'}^{m,m'}$ is the set of possible degrees $\beta_e$ for which $e$ is an edge in a localization graph adjacent to vertices $v$ and $v'$ with $k_v=k$, $k_{v'}=k'$, $m_{e,v}=m$, and $m_{e,v'}=m'$.

For $\beta\in E^{m,m'}_{k,k'}$, define the {\it recursive term} 
\[
\RC_{k,k'}^{m,m'}(\beta):=\frac{1}{d_{k'}\beta}\frac{\prod_{i}\prod_{0\leq b<-\beta w_i \atop \langle b \rangle =\langle mw_i \rangle}\left(\frac{b}{\beta}(\alpha_{k'}-\alpha_{k})-w_i\alpha_{k}\right)}{\prod_{j}\prod_{0 < b\leq-\beta d_j \atop \langle b \rangle =\langle md_j \rangle}'\left(\frac{b}{\beta}(\alpha_{k}-\alpha_{k'})+d_j(\alpha_{k}-\alpha_j)\right)}.
\]
Notice that the recursive term is equal to $d_ke_T(N_k^m)\;\Contr_{\Gamma}^W(e)$, where $e$ is an edge in a localization graph as above and $N_k^m$ is defined in Section \ref{flag}. For notational convenience, set
\[
\alpha_{k,k'}^\beta:=\frac{\alpha_{k'}-\alpha_{k}}{\beta}.
\]

\begin{theorem}\label{thm:conechar} Let $\bff \in \V_{\T}^X[[x]]$ be such that $(\bff|_{x=0})|_{Q=\infty}=0$. Then $\bff$ is a $\tilde\Lambda_{\nov}^{\T}[[x]]$-valued point of $\hat\cL_{\T}^{X,W}$ if and only if the following conditions hold:
\begin{enumerate}

\item[(C1)] For each $k,m$, the restriction $\bff_{k,m}$ lies in $\C(z,\alpha)((Q^{-\frac{1}{d}}))[[x]]$ and, as a rational function of $z$, each coefficient of a monomial in $Q$ and $x$ is regular except possibly for a pole at $z=0$, a pole at $z=\infty$, and poles at $z=\alpha_{k,k'}^\beta$ with $\beta\in E_{k,k'}^{m,m'}$ for some $k',m'$.

\item[(C2)] For each $k\neq k'$, $m$, $m'$, and $\beta\in E_{k,k'}^{m,m'}$, we have the following recursion:
\[
\mathrm{Res}_{z=\alpha_{k,k'}^\beta}\bff_{k,m}=Q^\beta \RC_{k,k'}^{m,m'}(\beta) \;\bff_{k',-m'}\big|_{z=\alpha_{k,k'}^\beta}.
\]

\item[(C3)] The Laurent expansion of $\bff_k$ at $z=0$ is a $\tilde\Lambda_{\nov}^{\T}[[x]]$-valued point of $\hat\cL^{P_k,W}_{c_k}\subset \V^{P_k}_{c_k}$.

\end{enumerate}
\end{theorem}

\begin{proof}
The proof follows that of Theorem 41 in \cite{CCIT2}.

Let $\bff$ be a $\tilde\Lambda_{\nov}^{\T}[[x]]$-valued point of $\hat\cL_{\T}^{X,W}$. We first verify that $\bff$ satisfies (C1) -- (C3). By definition, we can write $\bff$ as a formal series
\begin{equation}\label{eq:genlag}
I_{\T}^{X,W}(Q,-z) + \bt(z) +\sum_{n,\beta,\mu} \frac{Q^{\beta}}{n!}\left\langle \bt(\psi)^{n}\cdot \frac{\Phi_\mu}{-z-\psi}\right\rangle^{X,W,\T}_{0,n+1,\beta} \Phi^\mu,
\end{equation}
where $\bt(z)\in\V_{\T}^{X+}[[x]]$ satisfies $\bt(z)|_{x=0}=0$. The restriction $\bff_{k,m}$ can thus be written
\begin{equation}\label{eq:restrict}
I_{c_k}^{P_k,W}(Q,-z)_m+\bt_{k,m}(z)+\sum_{n,\beta} \frac{Q^{\beta}}{n!}\left\langle \bt(\psi)^n\cdot \frac{ (\one_{(m)}^k)^{\vee}}{-z-\psi}\right\rangle^{X,W,\T}_{0,n+1,\beta}\one_{(m)}^k,
\end{equation}
where $I_{c_k}^{P_k,W}(Q,-z)_m$ is the coefficient of $\one^k_{(m)}$ in the twisted $I$-function and the dual is taken with respect to the pairing $(-,-)_{c_k}^{P_k}$:
\begin{equation}
(\one_{(m)}^k)^{\vee}=d_ke_T(N_k^m)\one^k_{(-m)}.
\end{equation}
By Remark \ref{twistedI}, the initial term in \eqref{eq:restrict} is equal to 
\begin{equation}\label{eq:localI}
I_{c_k}^{P_k,W}(Q,-z)_m=z\sum_{l\in \Z_{>0}}Q^{-m-l}\frac{\prod_i\prod_{0\leq b<(m+l)w_i \atop \langle b \rangle = \langle mw_i \rangle}(bz-w_i\alpha_k)}{\prod_{j}\prod_{0<b<(m+l)d_j \atop \langle b \rangle = \langle md_j \rangle}(-bz+d_j(\alpha_k-\alpha_j))}
\end{equation}

The correlators in \eqref{eq:restrict} can be computed via the localization procedure outlined in Section \ref{localization}. Each localization graph has a distinguished vertex $v$ corresponding to the component carrying the last marked point, and the graphs subdivide into two types:
\begin{enumerate}
\item[A:] Graphs for which $v$ is unstable, i.e. $\val(v)=2$ and $\beta_v=0$;
\item[B:] Graphs for which $v$ is stable, i.e. $\val(v)>2$ or $\beta_v<0$.
\end{enumerate}

We now verify condition (C1).  It is clear from \eqref{eq:localI} that the initial term $I_{c_k}^{P_k,W}(Q,-z)_m$ lies in $\C(z,\alpha)((Q^{-\frac{1}{d}}))$ and it has poles at $z=0$, $z=\infty$, and $z=\frac{d_{k'}(\alpha_k-\alpha_{k'})}{b}$ where 
\[
b=md_{k'}+c>0
\]
for some integer $c$. Setting $\beta=-b/d_{k'}$ and $m'=\langle c/d_{k'}\rangle$, we see that these poles coincide with $\alpha_{k,k'}^\beta$ for $\beta\in E_{k,k'}^{m,m'}$.

Now consider the sum in \eqref{eq:restrict}. It follows from the virtual localization formula that this term lies in $\C(z,\alpha)((Q^{-\frac{1}{d}}))[[x]]$. Moreover, we saw in Section \ref{localization} that contributions from graphs of type A have the prescribed poles at $z=\alpha_{k,k'}^\beta$ due to the specialization of $\psi_{n+1}$ at the unstable vertices, while contributions from graphs of type B are polynomial in $z^{-1}$ because $\psi$ is nilpotent at the stable vertices. These observations prove (C1).

Next, we verify condition (C2). As before, we begin with the initial term $I_{c_k}^{P_k,W}(Q,-z)_m$. We compute directly that the residue of $I_{c_k}^{P_k,W}(Q,-z)_m$ at $z=\alpha_{k,k'}^\beta$ is equal to
\begin{equation}\label{Ires}
\frac{\alpha_{k'}-\alpha_k}{d_{k'}\beta^2}\hspace{-.2cm}\sum_{l\in\Z_{>0} \atop m+l>-\beta}Q^{-m-l}\frac{\prod_i\prod_{0\leq b < (m+l)w_i \atop \langle b\rangle = \langle mw_i \rangle}\left(\frac{b}{\beta}(\alpha_{k'}-\alpha_k)-w_i\alpha_k\right)}{\prod_j\prod_{0<b<(m+l)d_j \atop \langle b \rangle = \langle md_j\rangle}'\left(\frac{b}{\beta}(\alpha_k-\alpha_{k'})+d_j(\alpha_k-\alpha_j)\right)},
\end{equation}
and the evaluation of $I_{c_{k'}}^{P_{k'},W}(Q,-z)_{-m'}$ at $z=\alpha_{k,k'}^\beta$ is
\begin{equation}\label{Ieval}
\frac{\alpha_{k'}-\alpha_k}{\beta}\sum_{l\in\Z_{>0}}Q^{-\langle-m'\rangle-l}\frac{\prod_i\prod_{0\leq b<(\langle-m'\rangle+l)w_i \atop \langle b \rangle = \langle -m'w_i \rangle}(\frac{b}{\beta}(\alpha_{k'}-\alpha_k)-w_i\alpha_{k'})}{\prod_{j}\prod_{0<b<(\langle-m'\rangle+l)d_j \atop \langle b \rangle = \langle -m'd_j \rangle}(\frac{b}{\beta}(\alpha_k-\alpha_{k'})+d_j(\alpha_{k'}-\alpha_j))}.
\end{equation}
By shifting the index $b$ in \eqref{Ieval} so that the products start at $-\beta w_i$ in the numerator and  $-\beta d_j$ in the denominator, it is straightforward to verify that
\[
\eqref{Ires}=Q^\beta\;\RC_{k,k'}^{m,m'}(\beta)\cdot \eqref{Ieval}.
\]

Now we verify condition (C2) for the sum of correlators in \eqref{eq:restrict}. Let $\Gamma$ be a graph of type A. Then there is a unique edge $e$ adjacent to the (unstable) distinguished vertex $v$, and it meets the rest of the graph at a vertex $v'$. The contribution of $\Gamma$ to the particular correlator 
\[
\left\langle \bt^{n}(\psi)\; \frac{ (\one_{(m)}^k)^{\vee}}{-z-\psi}\right\rangle^{X,W,\T}_{0,n+1,\beta}
\]
can be written as
\[
\Contr^W_{\Gamma}=d_ke_T(N_k^m)\Contr^W_\Gamma(e)\;\Contr^W_{\Gamma'},
\]
where $\Gamma'$ is the graph obtained from $\Gamma$ by omitting the edge $e$ and $\Contr_{\Gamma'}$ is the contribution of $\Gamma'$ to the correlator
\[
\left\langle \bt^{n}(\psi)\cdot \frac{ (\one_{(m')}^{k'})^{\vee}}{-z-\psi}\right\rangle^{X,W,\T}_{0,n+1,\beta-\beta_e}.
\]
Since $\RC_{k,k'}^{m,m'}(\beta_e)=d_ke_T(N_k^m)\Contr^W_\Gamma(e)$, condition (C2) follows by fixing $e$ and summing over all possible $\Gamma'$.

Lastly, we verify (C3). Define
\[
\hat\bff_{k,m}:=\bff_{k,m}-I_{c_k}^{P_k,W}(Q,-z)_m
\]
and set
\[
\hat\bt_{k,m}(z):=\bt_{k,m}(z)+\sum_{k',m' \atop \beta\in E_{k,k'}^{m,m'}}\frac{Q^\beta \RC_{k,k'}^{m,m'}(\beta)}{z-\alpha_{k,k'}^\beta}\left(\hat\bff_{k',-m'}|_{z=\alpha_{k,k'}^\beta}\right),
\]
viewed as a power series at $z=0$. Notice that $\hat\bt_{k,m}(z)|_{x=0}=0$. From (C2), the second term is equal to the contribution from type-A graphs to the sum of correlators in \eqref{eq:restrict}. We now consider type-B graphs. By integrating over all moduli spaces except the one corresponding to the distinguished vertex, we compute that the contribution from all type-B graphs to the sum in \eqref{eq:restrict} is equal to
\[
\sum_{n,\beta} \frac{Q^{\beta}}{n!}\left\langle \hat\bt_{k,m}(\psi)^n\cdot \frac{(\one_{(m)}^k)^{\vee}}{-z-\psi}\right\rangle^{P_k,W,\epsilon=0,c_k}_{0,n+1,\beta}.
\]
Adding the type-A and type-B contributions and summing over $m$, we conclude that
\[
\bff_{k}=I_{c_k}^{P_k,W}(Q,-z)+\hat\bt_{k}(z)+\sum_{n,\beta,m} \frac{Q^{\beta}}{n!}\left\langle \hat\bt_{k}(\psi)^n\cdot \frac{\one_{(m)}^k}{-z-\psi}\right\rangle^{P_k,W,\epsilon=0,c_k}_{0,n+1,\beta}\hspace{-.4cm}(\one_{(m)}^k)^{\vee},
\]
which is a $\tilde\Lambda_{\nov}^{\T}[[x]]$-valued point of $\hat\cL^{P_k,W}_{c_k}$. This proves (C3).

To prove the reverse implication, assume $\bff$ satisfies (C1) -- (C3). As before, write 
\[
\hat\bff:=\bff-I_\T^{X,W}(Q,-z).
\]
Since both $\bff$ and $I_\T^{X,W}(Q,-z)$ satisfy conditions (C1) and (C2), so does $\hat\bff$, and we can write
\begin{equation}\label{eq:recurse}
\hat\bff_k=\bt_{k}(z)+\sum_{m,k',m' \atop \beta\in E_{k,k'}^{m,m'}}\frac{Q^\beta \RC_{k,k'}^{m,m'}(\beta)}{z-\alpha_{k,k'}^\beta}\left(\hat\bff_{k',-m'}|_{z=\alpha_{k,k'}^\beta}\right)\one_{(m)}^k
+O(z^{-1})
\end{equation}
for some $\bt_k(z)\in\V_{c_k}^{P_k,+}[[x]]$. Moreover, by condition (C3), we know that $\hat\bff|_{x=0}=0$. Choose $\bt(z)\in\V_{\T}^{X,+}[[x]]$ to be the unique element which restricts to $\bt_k(z)$ for all $k$. Then $\bff$ and the series \eqref{eq:genlag} both satisfy conditions (C1) -- (C3) and they give rise to the same restrictions $\bt_k(z)$. Therefore, it suffices to prove that (C1) -- (C3) uniquely determine $\bff$ from the collection $\bt_k(z)$. To justify this last claim, we proceed by lexicographic induction on the degree in $(x,Q^{-1})$.

Suppose we know the $x^{\alpha}Q^{\beta}$-coefficient of $\bff$ whenever $(\alpha, -\beta) <_{\text{lex}} (\mu, -\nu)$.  Since $\beta<0$ in the recursive term of \eqref{eq:recurse}, we can inductively determine the $x^{\mu}Q^{\nu}$-coefficient up to the $O(z^{-1})$ part. To determine this principal part, we use the fact that $\bff_k$ lies on $\hat\cL^{P_k,W}_{c_k}$, so we can write it as
\begin{equation}\label{twisteroo}
I_{c_k}^{P_k,W}(Q,-z) + \hat\bt_k(z) +\sum_{n,\beta,m} \frac{Q^\beta}{n!}\left\langle \hat\bt_k(\psi)^n\; \frac{\one_{(m)}^k}{-z-\psi}\right\rangle^{P_k,W,\epsilon=0,c_k}_{0,n+1,\beta}\left(\one_{(m)}^k\right)^\vee,
\end{equation}
in which
\[
\hat\bt_{k}(z):=\bt_{k}(z)+\sum_{m,k',m' \atop \beta\in E_{k,k'}^{m,m'}}\frac{Q^\beta \RC_{k,k'}^{m,m'}(\beta)}{z-\alpha_{k,k'}^\beta}\left(\hat\bff_{k',-m'}|_{z=\alpha_{k,k'}^\beta}\right)\one_{(m)}^k.
\]
Notice that the correlators in \eqref{twisteroo} capture the $O(z^{-1})$ part of $\hat\bff_k$ and they are nonzero only if $n>1$ or $n=1$ and $\beta<0$. Therefore, the inductive step also allows us to solve for the $O(z^{-1})$ part.
\end{proof}

\subsection{Proof of Theorem \ref{maintheorem}}

We now collect the results from the previous sections to prove Theorem \ref{maintheorem}--- that is, we prove that $\hat\cL_{\T}^{X,W}$ is a formal germ of the GW Lagrangian cone $\cL_{\T}^X$.

\begin{proof}[Proof of Theorem \ref{maintheorem}]

Since $\cL_{\T}^X$ is spanned by the derivatives of the $J$-function as in \eqref{eq:ruling}, we simply need to show that every point in the formal subspace $\hat\cL_{\T}^{X,W}$ can be written as a linear combination of derivatives of $J_\T^X(\tau,-z)$. By definition, a point of $\hat\cL_{\T}^{X,W}$ can be written as
\[
\bff=I_\T^{X,W}(Q,-z)+\bt(z) +\sum_{n,\beta,\mu} \frac{Q^\beta}{n!}\left\langle \bt^{n}(\psi)\; \frac{\Phi_\mu}{-z-\psi}\right\rangle^{X,W,\T}_{0,n+1,\beta} \Phi^\mu.
\]
We can find an element of $\cL_\T^X$ that matches $\bff$ modulo $\V^{X,-}_\T$, since the GW Lagrangian cone $\cL_\T^X$ is a graph over $\V^{X,+}_\T$.  Written in terms of derivatives of the $J$-function, this means there is a unique
\[
\bfg={\hat\bt(z)}+\sum_{n,\beta,\mu}\frac{Q^\beta}{n!}\left\langle {\hat\bt(z)}\;\tau^n\;\frac{\Phi_\mu}{-z-\psi} \right\rangle_{0,n+2,\beta}^{X,\T}\Phi^\mu\in \cL_\T^X
\]
such that $\bff=\bfg \mod z^{-1}$. Here, both ${\hat\bt(z)}\in z\V_\T^{X,+}[[t]]$ and $\tau$ depend formally on $\bt(z)$. To prove that $\bff\in\cL_\T^X$, we must prove that $\bff=\bfg$. Since points of $\hat\cL_{\T}^{X,W}$ are uniquely determined by their projection to $\V_{\T}^{X,+}$, we can prove that $\bff=\bfg$ by showing that $\bfg$ is a $\tilde\Lambda_{\nov}^{\T}[[t]]$-valued point of $\hat\cL_{\T}^{X,W}$. We accomplish this by verifying conditions (C1) -- (C3) of Theorem \ref{thm:conechar}.

First, notice that $(\bff|_{t=0})|_{Q=\infty}=0$ implies that the same is true of $\bfg$. Now consider the restriction of $\bfg$ to the span of $\one_{(m)}^k$:
\[
\bfg_{k,m}=\hat\bt_{k,m}(z)+\sum_{n,\beta}\frac{Q^\beta}{n!}\left\langle \hat\bt(z)\;\tau^n\;\frac{\left(\one_{(m)}^k \right)^\vee}{-z-\psi}\right\rangle_{0,n+2,\beta}^{X,\T}.
\]
By the virtual localization formula in GW theory, each correlator in $\bfg_{k,m}$ can be computed as a sum over contributions from localization graphs. Each localization graph has a distinguished vertex $v$ supporting the last marked point and the graphs split into two types:
\begin{enumerate}
\item[A:] Graphs for which $\val(v)=2$ and $\beta_v=0$;
\item[B:] Graphs for which $\val(v)>2$ or $\beta_v<0$.
\end{enumerate}

As in the proof of Theorem \ref{thm:conechar}, the contributions from type-A graphs have poles at $z=\alpha_{k,k'}^\beta$ while the contributions from type-B graphs have poles at $z=0$. This proves (C1). Using the fact that $\Contr_\Gamma(e)=\Contr_\Gamma^W(e)$, the same analysis used in the proof of Theorem \ref{thm:conechar} proves that $\bfg_{k,m}$ satisfies the recursions described by condition (C2).

It is left to prove (C3)--- that is, that $\bfg_{k}$ is a $\tilde\Lambda_{\nov}^{\T}[[x]]$-valued point of $\hat\cL_{c_k}^{P_k,W}$. To do this, consider the series
\[
\tilde\bfg:=-\one z+\tau+\sum_{n,\beta,\mu}\frac{Q^\beta}{n!}\left\langle \tau^n\;\frac{\Phi_\mu}{-z-\psi}\right\rangle_{0,n+1,\beta}^{X,\T}\Phi^\mu.
\]
We can write
\[
\tilde\bfg_k=-\one z+\tilde\tau_k(z)+\sum_{n,\beta,m}\frac{Q^\beta}{n!}\left\langle \tilde\tau_k(\psi)^n\;\frac{\one_{(m)}^k}{-z-\psi}\right\rangle_{0,n+1,\beta}^{P_k,c_k}\left(\one_{(m)}^k\right)^\vee,
\]
where 
\[
\tilde\tau_k(z):=\tau_{k}+\sum_{\Gamma \text{ of Type A }}\Contr_\Gamma,
\]
viewed as a power series in $z$.  From this, we see that $\tilde\bfg_k$ is an element of $\cL_{c_k}^{P_k}$.  Let $\partial_{\hat\bt(z)}$ be the differential operator defined by replacing $\Phi_{\mu}$ by $\frac{\d}{\d \tau^{\mu}}$ in the definition of $\hat\bt(z)$, so that $\partial_{\hat\bt(z)} \tau= \hat\bt(z)$.  Then
\[
\bfg_k=\partial_{\hat\bt(z)}\tilde\bfg_k.
\]
Since $\cL_{c_k}^{P_k}$ is an over-ruled cone containing $\tilde\bfg_k$, it follows that $\bfg_k \in\cL_{c_k}^{P_k}$. 

Applying the twisted cone correspondence from Corollary \ref{cor:twisted}, we conclude that $\bfg_k\in\cL_{c_k}^{P_k,W}$. Moreover, since 
\[
\bfg_k|_{x=0} = \bff_k|_{x=0} = I_{c_k}^{P_k,W}(Q,-z) \mod z^{-1}
\]
and points on the twisted cone are determined by their regular part in $z$, we have $\bfg_k|_{x=0}=I_{c_k}^{P_k,W}(Q,-z)$. This implies that $\bfg_k$ is a  $\tilde\Lambda_{\nov}^{\T}[[x]]$-valued point of $\cL_{c_k}^{P_k,W}$ and finishes the proof.
\end{proof}

\section{Phase transitions}
\label{transitions}

In this last section, we use Theorem \ref{maintheorem}, along with previously-known results concerning the crepant transformation conjecture and quantum Serre duality, to deduce a correspondence between the gauged linear sigma models that arise at different phases of the GIT quotient.  In other words (recalling that the positive phase of the GLSM gives the GW theory of the complete intersection $Z$ cut out by the polynomials $F_j$), we identify the genus-zero GLSM of $(X_-,W)$ with the GW theory of $Z$. 

Throughout this section, we assume \textbf{(A1)}, \textbf{(A2)}, and the Calabi--Yau condition:
\[
\sum_{i=1}^M w_i=\sum_{j=1}^N d_j.
\]
We expect the results to extend to the non-Calabi--Yau case, following arguments developed by Acosta \cite{Acosta} and Acosta--Shoemaker \cite{AcostaShoemaker}.

\subsection{Notation}

Recall that
\[
X_-=X:=\bigoplus_{i=1}^M\O_{\P(\vec d)}(-w_i),
\]
\[
X_+:=\bigoplus_{j=1}^N\O_{\P(\vec w)}(-d_j),
\]
and $Z$ is the complete intersection 
\[
Z:=Z(F_1,\dots,F_N)\subset\P(\vec w)\subset X_+.
\]

We have $H_{CR}^*(X_-)=H_{CR}^*(\P(\vec d))$ and  $H_{CR}^*(X_+)=H_{CR}^*(\P(\vec w))$. It is a standard fact that
\[\text{rank} \big(H^*_{CR}(\P(\vec d))\big) = \sum d_j,\]
so the Calabi--Yau condition implies that there is a vector space isomorphism
\begin{equation}\label{eq:vs1}
H_{CR}^*(X_-)\cong H_{CR}^*(X_+).
\end{equation}
We simultaneously choose bases for $H_{CR}^*(X_\pm)$ by declaring \[H_{(m)}^l:=e\left(\O_{X_{(m)}}(l)\right)\] regardless of the GIT phase.

These are not, strictly speaking, the state spaces of the GLSM in the two phases; recall, the GLSM state space is defined as $H^*_{CR}(X_{\pm}, W^{+\infty})$, where $W^{+\infty}$ is a Milnor fiber.  Nevertheless, as we have seen, $H^*_{CR}(X_-)$ contains the narrow part of the state space in the negative phase,
\[\H^W\subset H_{CR}^*(X_-),\]
which is generated by $H_{(m)}^l$ with $m\in\mathrm{nar}$.  Analogously, the GLSM state space in the positive phase is isomorphic to $H^*_{CR}(Z)$ and contains the ambient part
\[
\H^Z\subset H_{CR}^*(Z),
\] 
which is defined as the image of $i^*:H^*_{CR}(X_+) \rightarrow H_{CR}^*(Z)$, where $i:Z\rightarrow X_+$ is the inclusion. It follows from assumption \textbf{(A2)} that the vector space isomorphism \eqref{eq:vs1} induces a vector space isomorphism\footnote{We thank Pedro Acosta for pointing out the necessity of assumption (\textbf{A2}) in this isomorphism.} 
\begin{equation}\label{eq:vs2}
\H^W\cong \H^Z.
\end{equation}
This is a special case of the state space isomorphism proved by Chiodo--Nagel \cite{CN}.

\subsection{Crepant transformation conjecture}

The crepant transformation conjecture identifies the GW theory of two targets related by a crepant birational transformation (see, for example, \cite{CoatesRuan}). Recently, Coates--Iritani--Jiang proved the crepant transformation conjecture for a large class of toric targets \cite{CIJ}. Their results include, as a special case, the phase transition between the GW theories of $X_-$ and $X_+$: 

\begin{theorem}[\cite{CIJ}, Theorem 6.1]\label{thm:ctc}
Let $\cL_\T^{X_\pm}\subset\V_\T^{X_\pm}$ be the Lagrangian cones associated to the $\T$-equivariant GW theory of $X_\pm$. There exists a $\C(\alpha,z)$-linear symplectomorphism $\bU_\T:\V_\T^{X_-}\rightarrow \V_\T^{X_+}$ such that
\begin{enumerate}
\item $\bU_\T$ matches Lagrangian cones after substituting $Q=1$ and analytic continuation:
\[
\bU_\T(\cL_\T^{X_-})=\cL_\T^{X_+};
\] 
\item $\bU_\T$ is induced by a Fourier--Mukai transformation 
\[
\FM:K_\T^0(X_-)\rightarrow K_\T^0(X_+)
\]
via a diagram of the form
\[
\xymatrixcolsep{4pc}
\xymatrix{
	K_\T^0(X_-) \ar[r]^{\FM} \ar[d]^{\tilde{\Psi}_-} & K_\T^0(X_+)\ar[d]^{\tilde{\Psi}_+}\\
	\V_\T^{X_-} \ar[r]^{\bU_\T} & \V_\T^{X_-}.
	}
	\]
\end{enumerate}
\end{theorem}

It will be useful in what follows to unravel Theorem \ref{thm:ctc}. In particular, we describe how $\bU_\T$ is induced from the Fourier--Mukai transformation. To give such a description, we recall from \cite{CIJ} that the maps $\tilde\Psi_{\pm}$ appearing in Theorem \ref{thm:ctc} are defined by 
\[
\tilde\Psi_{\pm}(E)=z^{-\mu^{\pm}}z^{\rho^{\pm}}\left(\hat\Gamma_{X_{\pm}}\cup(2\pi\ri)^{\frac{\deg_0}{2}}\mathrm{inv}^*\ch_\T(E)\right)
\]
and $\bU_\T=\tilde\Psi_{+}\circ\FM\circ\tilde\Psi_{-}^{-1}$. Rather than recalling all of the notation from \cite{CIJ}, we content ourselves with observing that we can write
\[
\bU_\T=\Gamma_+\circ\overline{\bU}_\T\circ\Gamma_{-}^{-1},
\]
where
\[
\overline\bU_{\T}=\ch_\T\circ\FM\circ\ch_\T^{-1}
\]
and the maps $\Gamma_{\pm}$ act diagonally on the sectors of the inertia stack and have well-defined and invertible non-equivariant limits. As a consequence of this structure, everything we need to prove about $\bU_\T$ can be proved by understanding $\overline\bU_\T$.

To describe $\overline \bU_\T$ explicitly, we first note that generators for $K_0^G(\C^{M+N})$ are given by the line bundles $L_\rho$, which are geometrically trivial and have $G$-linearization of weight $\rho$. Each such line bundle induces line bundles $L_\rho^+$ and $L_\rho^-$ on $X_+$ and $X_-$, respectively, and the Fourier-Mukai morphism is defined by by $\FM(L_\rho^-)=L_\rho^+$. We compute
\[
\ch_\T(L_\rho^+)=\sum_m\re^{2\pi\ri\rho m}\re^{\rho H}\one_{(m)}
\]
and
\[
\ch_\T(L_\rho^-)=\sum_{m,k}\re^{-2\pi\ri\rho m}\re^{\rho\alpha_k}\one_{(m)}^k,
\]
where we have chosen to write the Chern characters on $X_-$ in terms of the localized basis. Allowing $\rho$ to vary between $0$ and $D-1$, where 
\[
D:=\mathrm{rk}(H_{CR}^*(X_-))=\sum_j d_j=\sum_i w_i,
\] 
we see that the map $\ch_\T(L_\rho^-)$ is given by a Vandermonde matrix:
\[
\left(\ch_\T(L_\rho^-)\right)_{\rho}=\left(x_{k,m}^\rho\right)_{\rho}^{k,m}\left(\one_{(m)}^k\right)_{k,m}.
\]
Here, $x_{k,m}:=\re^{-2\pi\ri m+\alpha_k}$, and the lower and upper indices denote rows and columns of a matrix, respectively.

Inverting the Vandermonde matrix, we compute that
\[
\left(\one_{(m)}^k\right)_{k,m}=\left((-1)^{D-\rho-1}\frac{\sum_{S}\prod_{(k',m')\in S} x_{k',m'}}{\prod_{(k',m')\neq(k,m)}(x_{k,m}-x_{k',m'})} \right)_{k,m}^{\rho}\left(\ch_\T(L_\rho^-)\right)_{\rho},
\]
where the sum in the numerator is over all sets of pairs $(k_i, m_i) \neq (k.m)$ of size $D-\rho-1$:
\[
S=\left\{(k_1,m_1),\dots,(k_{D-\rho-1},m_{D-\rho-1})\;|\;(k_i,m_i)\neq(k,m)\right\}.
\]
In particular, we compute:
\begin{align}\label{chern}
\nonumber \overline\bU_\T(\one_{(m)}^k)&=\sum_{0\leq \rho<D \atop l}(-1)^{D-\rho-1}\frac{\sum_{S}\prod_{(k',m')\in S} x_{k',m'}}{\prod_{(k',m')\neq(k,m)}(x_{k,m}-x_{k',m'})}\re^{2\pi\ri\rho l}\re^{\rho H}\one_{(l)}\\
&=\sum_{0\leq \rho<D \atop l}(-1)^{D-\rho-1}\frac{\sum_{S}\prod_{(k',m')\in S} y_{k',m',l}}{\prod_{(k',m')\neq(k,m)}(y_{k,m,l}-y_{k',m',l})}\one_{(l)},
\end{align}
where $y_{k,m,l}:=x_{k,m}\re^{-2\pi\ri l-H}=\re^{-2\pi\ri (m+l)+\alpha_k-H}$. 

The main structural results that we need concerning $\overline\bU_\T$ are contained in the following lemma.

\begin{lemma}\label{lem:vanish}
The $\one_{(l)}$-coefficient of $\overline\bU_\T(\one_{(m)}^k)$ satisfies the following properties:
\begin{enumerate}[(i)]
\item it vanishes at $H=\alpha_j$ whenever $l\neq -m$ or $j\neq k$, and
\item it has simple poles at $\alpha_j=\alpha_k$ for $j\neq k$.
\end{enumerate}
\end{lemma}

\begin{proof}
We begin by proving the first assertion. Notice that, after the evaluation $H = \alpha_j$, we have $y_{j,-l,l}=1$.  This allows us to identify the (oppositely-signed) summands of \eqref{chern} indexed by $S$ and
\[
S':=\begin{cases}
S\cup\{(j,-l)\} &\text{ if } (j,-l)\notin S\\
S\setminus\{(j,-l)\} &\text{ if } (j,-l)\in S,
\end{cases}
\]
as long as $(j,-l)\neq (k,m)$.

To prove the second assertion, notice that
\[
y_{k,m,l}-y_{k',m',l}=\re^{-2\pi\ri(m+l)+\alpha_k-H}\left(1- \re^{-2\pi\ri(m'-m)+\alpha_{k'}-\alpha_k} \right),
\]
which vanishes linearly at $\alpha_{k'}=\alpha_k$ when $m'=m$ and $k'\neq k$.
\end{proof}

\subsection{Quantum Serre and Lefschetz}

Quantum Serre duality, developed by Coates--Givental \cite{CoatesGivental} for varieties and Tseng \cite{Tseng} for orbifolds, can be used to relate the genus-zero GW invariants of $\P(\vec w)$ twisted by the $\T$-equivariant inverse Euler class of $\bigoplus_j\O(-d_j)$ to the genus-zero GW invariants of $\P(\vec w)$ twisted by the $\T$-equivariant Euler class of the dual bundle $\bigoplus_j\O(d_j)$ (with dual $\T$-action). The former invariants are the $\T$-equivariant GW invariants of $X_+$. The quantum Lefschetz theorem states that the non-equivariant limit of the latter invariants can be related to the ambient part of the genus-zero GW invariants of the complete intersection $Z$.

We start with quantum Serre duality:

\begin{theorem}[Tseng \cite{Tseng}]\label{thm:qsd}
Let $\cL_\T^{X_+}\subset \V_\T^{X_+}$ be the Lagrangian cone associated to the GW theory of $\P(\vec w)$ twisted by the $\T$-equivariant Euler class of $\bigoplus_j\O(-d_j)$, and let $\cL_\T^{\P(\vec w),e}\subset \V_\T^{\P(\vec w),e}$ be the Lagrangian cone associated to the GW theory of $\P(\vec w)$ twisted by the $\T$-equivariant Euler class of $\bigoplus_j\O(d_j)$. Then the symplectomorphism $\phi_{\T}^+:\V_\T^{X_+}\rightarrow\V_\T^{\P(\vec w),e}$ defined by
\[
H_{(m)}^l\mapsto \frac{\re^{\pi\ri\sum_j\left(\left\langle d_j m\right\rangle-d_jH_{(m)}/z\right)}}{e_\T\left(\bigoplus_{j=1}^N\O(d_j) \right)}H_{(m)}^l
\]
identifies $\cL_\T^{X_+}$ with $\cL_\T^{\P(\vec w),e}$.
\end{theorem}

We now recall the quantum Lefschetz theorem.  Denoting by $i$ the inclusion $Z 
\rightarrow X$, as above, the theorem can be rephrased in our setting as follows:

\begin{theorem}[Coates \cite{Coates2}]\label{thm:qlt}
Let $\bff$ be a point of $\cL_\T^{\P(\vec w),e}$ with a well-defined non-equivariant limit $\lim_{\alpha\rightarrow 0}\bff$. Then $\lim_{\alpha\rightarrow 0}i^*\bff$ lies on $\cL^Z$.
\end{theorem}

\begin{remark}
The condition \textbf{(A1)} is necessary here; it is equivalent to the assertion that $\bigoplus_j\O(d_j)$ is pulled back from the coarse underlying space of $\P(\vec w)$.
\end{remark}



\subsection{Narrow GLSM cone as a non-equivariant limit}

In all of the paper thus far, we have been working with the equivariantly-extended GLSM, and not the narrow GLSM, which was our original motivation. We now turn to the study of the narrow GLSM.

Define the Givental space associated to the narrow state space
\[
\V^{X_-,W}:=\H^W[z,z^{-1}]((Q^{-\frac{1}{d}})),
\]
and let the (non-equivariant) formal subspace $\hat\cL^{X_-,W}$ be the collection of points of the form
\begin{equation}
\label{eq:noneqsub}
I^{X_-,W}(Q,-z)+\bt(z) +\sum_{n,\beta\atop \mu\in\text{nar}} \frac{Q^\beta}{n!}\left\langle \bt(\psi)^n\; \frac{\Phi_\mu}{-z-\psi}\right\rangle^{X_-,W}_{0,n+1,\beta} \Phi^\mu,
\end{equation}
where $\bt(z)\in\V^{X_-,W,+}$, $\mu$ only varies over a basis of the narrow sectors, and $I^{X_-,W}(Q,z)$ is the non-equivariant limit of $I^{X_-,W}_\T(Q,z)$.

The equivariant and non-equivariant formal subspaces are related as follows:

\begin{lemma}\label{nelimit}
The formal subspace $\hat\cL^{X_-,W}$ lies in $\V^{X_-,W}$ and can be obtained from $\hat\cL_\T^{X_-,W}$ by first restricting $\bt(z)$ to $\V^{X_-,W,+}\subset\V_\T^{X_-,+}$ and then taking a non-equivariant limit.\footnote{There is a slight abuse of terminology here. By the ``restriction'' of $H_{CR,T}^*(X_-)$ to $\H^W$, we mean the restriction to equivariant cohomology classes that
\begin{enumerate}[(a)]
\item are supported on the narrow sectors, and
\item have well-defined non-equivariant limits in $\H^W$.
\end{enumerate}
}
\end{lemma}

\begin{proof}
To prove the first assertion, we only need to show that 
\[
I^{X_-,W}(Q,z)=z\sum_{a\in\frac{1}{d}\Z \atop a>0} Q^{-a}\frac{\prod_{i=1}^M\prod_{0\leq b< a w_i \atop \langle b \rangle = \langle a w_i \rangle}(-bz+w_iH)}{\prod_{j=1}^N\prod_{0< b< a d_j \atop \langle b \rangle = \langle a d_j \rangle}(bz+d_jH)}\one_{( a )}
\]
lies in the narrow Givental space. Suppose $\langle a\rangle\notin\text{nar}$, so that $w_{i_0}a\in\Z$ for some $i_0$.  Assumption \textbf{(A2)} asserts that, setting $I=\{i:w_ia\in\Z\}$ and $J=\{j:d_ja\in\Z\}$, we have $|I|\geq |J|$. In particular, this implies that the numerator of $I^{X_-,W}(Q,z)$ has a factor of $H_{(a)}^{|J|}=0$.

Now consider the second assertion.  A point in $\hat\cL_\T^{X_-,W}$ for which $\t(z) \in \V^{X_-,W,+}$ indeed has a well-defined non-equivariant limit, because all but possibly one of the insertions in the correlators that appear are drawn from the narrow state space; thus, by Remark \ref{rmk:onebroad}, the virtual class is an Euler class of a vector bundle, so it admits a non-equivariant limit.  Moreover, the summand of \eqref{eq:formalGLSM} indexed by $\mu$ vanishes whenever $\Phi_\mu\notin\H^W$. Indeed, if $\Phi_\mu=H_{(m)}^l$ with $m\notin\text{nar}$, then we have
\[
\Phi^\mu=H_{(-m)}^{|J|-l}e_\T\left(\bigoplus_{i\in I}\O_{X_{(m)}}(w_i) \right);
\]
here, as above, $I=\{i:w_im\in\Z\}$ and $J=\{j:d_jm\in\Z\}$. Assumption \textbf{(A2)} asserts that $|I|\geq |J|$ and it follows that $\lim_{\alpha\rightarrow 0}\Phi^\mu$ has a factor of $H_{(-m)}^J=0$.  This proves that the non-equivariant limit of a point in $\hat\cL_\T^{X_-,W}$ with $\t(z) \in \V^{X_-,W,+}$ is of the form \eqref{eq:noneqsub}, as claimed.
\end{proof}

\begin{corollary}\label{cor:nelimit}
The subspace $\hat\cL^{X_-,W}$ is a formal germ of an over-ruled Lagrangian cone $\cL^{X_-,W}$. In particular, $\cL^{X_-,W}$ consists of points of the form
\[
\left\{ z\hat\bt(z)+\sum_{n,\beta \atop \mu\in\text{nar}} \frac{Q^\beta}{n!}\left\langle z\hat\bt(z)\;t^n\; \frac{\Phi_\mu}{-z-\psi}\right\rangle^{X_-,W}_{0,n+2,\beta} \Phi^\mu \; \Bigg|\; {\hat\bt(z)\in\V^{X_-,W,+}\atop t\in\H^W}\right\}.
\] 
In addition, all points of $\cL^{X_-,W}$ are obtained by taking the non-equivariant limit of points in a subspace $\cL_{\T,\mathrm{pre}}^{X_-,W}\subset\cL_\T^{X_-,W}$.
\end{corollary}

\begin{proof}
Since $\hat\cL^{X_-,W}_\T$ is a germ of an over-ruled Lagrangian cone by Theorem \ref{maintheorem}, the corresponding fact for $\hat\cL^{X_-,W}$ follows from Lemma \ref{nelimit}. This implies that the points of $\cL^{X_-,W}$ can be written as formal linear combinations of any transverse slice, leading to the description given in the statement of the corollary. 

To express $\cL^{X_-,W}$ as a non-equivariant limit, write a general point of the equivariant cone $\cL^{X_-,W}_\T$ as
\[
\left\{ z\hat\bt(z)+\sum_{n,\beta,\mu} \frac{Q^\beta}{n!}\left\langle z\hat\bt(z)\;t^n\; \frac{\Phi_\mu}{-z-\psi}\right\rangle^{X_-,W,\T}_{0,n+2,\beta} \Phi^\mu \; \Bigg| \;{\hat\bt(z)\in\V_\T^{X_-,+}\atop t\in H_{CR}^*(X_-)}\right\}
\]
and define $\cL_{\T,\mathrm{pre}}^{X_-,W}$ to be the set of points such that $\hat\bt(z)\in\V^{X_-,W,+}$ and $t\in\H^W$. Then the same argument given in the proof of Lemma \ref{nelimit} implies that 
\[
\cL^{X_-,W}=\lim_{\alpha\rightarrow 0}\cL_{\T,\mathrm{pre}}^{X_-,W}.
\]
\end{proof}

\begin{remark}
The ``pre'' in the notation stands for ``pre-narrow''. It indicates that $\cL_{\T,\mathrm{pre}}^{X_-,W}$ does not necessarily lie in the narrow subspace, but its non-equivariant limit does lie in the narrow subspace and recovers the narrow cone.
\end{remark}

\subsection{Proof of Theorem \ref{maintheorem2}}

We now outline the proof of Theorem \ref{maintheorem2}, leaving the details for Lemma \ref{lastlemma}.

The combination of Theorems \ref{maintheorem}, \ref{thm:ctc}, and \ref{thm:qsd} provides us with a symplectomorphism
\[
\phi_\T^+\circ\bU_\T:\V_\T^{X_-}\rightarrow \V_\T^{\P(\vec w),e}
\]
that identifies the extended GLSM Lagrangian cone $\cL_\T^{X_-,W}\subset \V_\T^{X_-}$ with the twisted GW Lagrangian cone $\cL_\T^{\P(\vec w),e}\subset \V_\T^{\P(\vec w),e}$, after analytic continuation.  (Recall, $\cL_\T^{X_-, W}$ is defined, in light of Theorem \ref{maintheorem}, to be the GW cone $\cL_\T^{X_-}$.)  In order to prove Theorem \ref{maintheorem2}, we need to investigate the non-equivariant limit of a suitable restriction of $\phi_\T^+\circ\bU_\T$. 

More specifically, we prove in Lemma \ref{lastlemma} below that the composition $\phi_\T^+\circ\bU_\T$ has a well-defined non-equivariant limit after restricting to the narrow subspace $\V^{X_-,W}$, and this allows us to define the symplectic isomorphism $\bV:\V^{X_-,W}\rightarrow\V^Z$  by
\[
\bV:=\lim_{\alpha\rightarrow 0}\left(i^*\circ\phi_\T^+\circ\bU_\T\big|_{\V^{X_-,W}}\right).
\]

In addition, we prove in Lemma \ref{lastlemma} that, for any $\bff\in \cL_{\T,\mathrm{pre}}^{X_-,W}$, the image $\phi_\T^+\circ\bU_\T\left(\bff \right)$ has a well-defined non-equivariant limit. Thus, by Theorem \ref{thm:qlt}, we obtain 
\begin{equation}\label{equationone}
\lim_{\alpha\rightarrow 0}\left(i^*\circ\phi_\T^+\circ\bU_\T\left(\bff\right)\right)\in \cL^Z.
\end{equation}

Lastly, we prove that, for any $\bff\in \cL_{\T,\mathrm{pre}}^{X_-,W}$, the non-equivariant limits commute:
\begin{equation}\label{equationtwo}
\lim_{\alpha\rightarrow 0}\left(i^*\circ\phi_\T^+\circ\bU_\T\left(\bff\right)\right)=\bV\left(\lim_{\alpha\rightarrow 0}\bff\right).
\end{equation}
Since all points of $\cL^{X_-,W}$ are obtained as $\lim_{\alpha\rightarrow 0}\bff$ for some $\bff\in \cL_{\T,\mathrm{pre}}^{X_-,W}$ (Corollary \ref{cor:nelimit}), equations \eqref{equationone} and \eqref{equationtwo} imply that $\bV$ identifies $\cL^{X_-,W}$ with $\cL^Z$.

The following lemma provides the requisite details to complete these arguments.

\begin{lemma}\label{lastlemma}
With definitions as above, we have the following:
\begin{enumerate}[(i)]
\item The restricted symplectomorphism $\phi_\T^+\circ\bU_\T\big|_{\V^{X_-,W}}$ has a well-defined non-equivariant limit, and the map
\[
\bV:=\lim_{\alpha\rightarrow 0}\left(i^*\circ\phi_\T^+\circ\bU_\T\big|_{\V^{X_-,W}}\right)
\]
is a symplectic isomorphism.
\item  For any $\bff\in \cL_{\T,\mathrm{pre}}^{X_-,W}$, the image $\phi_\T^+\circ\bU_\T\left(\bff \right)$ has a well-defined non-equivariant limit.
\item For any $\bff\in\cL_{\T,\mathrm{pre}}^{X_-,W}$, we have 
\[
\lim_{\alpha\rightarrow 0}\left(i^*\circ\phi_\T^+\circ\bU_\T\left(\bff \right)\right)=\bV\left(\lim_{\alpha\rightarrow 0}\bff \right).
\]
\end{enumerate}
\end{lemma}

\begin{proof}
We begin with assertion (i). We must show that $\phi_\T^+\circ\bU_\T(\Phi)$ has a non-equivariant limit whenever $\Phi\in\H^W$. If $\Phi=H_{(m)}^a$ with $m\in\text{nar}$, then the localization isomorphism allows us to write $\Phi$ in terms of the classes $\one_{(m)}^k$. Narrowness implies that $\one_{(-m)}=0\in H_{CR}^*(X_+)$ and, putting this together with Lemma \ref{lem:vanish}, we see that $\bU_\T(\Phi)$ vanishes at $H=\alpha_j$ for all $j$. Since $\bU_\T(\Phi)$ has a non-equivariant limit, this implies that we can write
\[
\bU_\T(\Phi)=\prod_{j=1}^N(H-\alpha_j)\widehat{\bU}_\T(\Phi),
\]
in which $\widehat{\bU}_\T(\Phi)$ has a non-equivariant limit.  The transformation $\phi_{\T}^+$ is defined via division by $e_\T\left(\oplus_j\O(d_j)\right)=\prod_jd_j(H-\alpha_j)$, so this implies that $\phi_\T^+\circ\bU_\T(\Phi)$ has a well-defined non-equivariant limit.  Moreover, the limit is manifestly supported away from the top $N$ powers of $H$, a fact we will use shortly.

We can now define $\bV:=\lim_{\alpha\rightarrow 0}\left(i^*\circ\phi_{\T}^+\circ\bU_\T|_{\V^{X_-,W}}\right)$. To see that $\bV:\V^{X_-,W}\rightarrow \V^Z$ is a symplectic isomorphism, note that the original map $\phi_\T^+\circ\bU_\T$ is a symplectic isomorphism, implying that its restriction to $\V^{X_-,W}$ is also a symplectic isomorphism onto its image. Taking non-equivariant limits, the fact that the image $\lim_{\alpha\rightarrow 0}\phi_\T^+\circ\bU_\T|_{\V^{X_-,W}}$ is supported away from the top $N$ powers of $H$ implies that $i^*$ identifies this image with the ambient part of $\V^Z$.

We now prove assertion (ii).  Since $\bff$ need not be supported on the narrow subspace, the same argument as above does not immediately apply.  However, by the definition of $\cL^{X_-,W}_{\T,\mathrm{pre}}$, the cohomology classes appearing in $\bff$ are of the form
\[
\Phi^\mu=H_{(m)}^a e_\T\left(\bigoplus_{i\in I}\O_{X_{(m)}}(w_i) \right),
\]
where $I=\{i:w_im\in\Z\}$.  We only consider the case where $\Phi^{\mu}$ is not narrow, since in the narrow situation, the proof of (i) does imply the existence of the non-equivariant limit of $\phi_\T^+ \circ \bU_\T(\Phi^{\mu})$.

By the localization isomorphism, we can write
\[
H_{(m)}^a=\sum_{k}f_{k}(\alpha)\one_{(m)}^k,
\]
in which each $f_k(\alpha)$ is a degree-$a$ homogeneous polynomial in the $\alpha_j$. By linearity, we have
\begin{equation}\label{eq:sumone}
\bU_\T(H_{(m)}^l)=\sum_{k}f_{k}(\alpha)\bU_\T\left(\one_{(m)}^k\right).
\end{equation}
Multiplying by the Euler class, we compute
\[
\Phi^\mu=\sum_{k}f_{k}(\alpha)\one_{(m)}^k\prod_{i\in I}w_i\alpha_k
\]
and
\begin{equation}\label{eq:sumtwo}
\bU_\T(\Phi^\mu)=\sum_{k}f_{k}(\alpha)\bU_\T\left(\one_{(m)}^k\right)\prod_{i\in I}w_i\alpha_k.
\end{equation}

Let $\bU_\T(\Phi)_{(l)}$ denote the part of $\bU_\T(\Phi)$ supported on the twisted sector indexed by $l$. By the same argument given in the proof of (i), the image $\phi_\T^+\left(\bU_\T(H_{(m)}^a)_{(l)}\right)$ has a well-defined non-equivariant limit as long as $l\neq -m$. Given that $\Phi^{\mu}$ is not narrow and hence $I \neq \emptyset$, one obtains $\bU_\T(\Phi^\mu)_{(l)}$ from $\bU_\T(H_{(m)}^a)_{(l)}$ by multiplying each summand in \eqref{eq:sumone} by a positive power of $\alpha$.  Thus,
\begin{equation}\label{brvanish}
\lim_{\alpha\rightarrow 0}\phi_\T^+\left(\bU_\T(\Phi^\mu)_{(l)}\right)=0 \text{ whenever }l\neq -m.
\end{equation}

It is left to prove that $\phi_\T^+\left(\bU_\T(\Phi^\mu)_{(-m)}\right)$ has a well-defined non-equivariant limit. By Lemma \ref{lem:vanish}, we know that the $k$-summand of \eqref{eq:sumtwo} has zeroes at $H=\alpha_j$ for $j\neq k$ and possible poles along $\alpha_k=\alpha_j$ for $j\neq k$. Since $\bU_\T(\Phi^\mu)$ has a well-defined non-equivariant limit, the poles cancel in the sum. Thus, we can write
\begin{align}\label{eq:sumthree}
\nonumber\frac{\bU_\T(\Phi^\mu)_{(-m)}}{\prod_{j=1}^N(H-\alpha_j)}&=\sum_{k}f_{k}(\alpha)\frac{\widehat\bU_\T\left(\one_{(m)}^k\right)_{(-m)}\prod_{i\in I}w_i\alpha_k}{H-\alpha_k}\\
&=-\sum_{k}f_{k}(\alpha)\widehat\bU_\T\left(\one_{(m)}^k\right)_{(-m)}\prod_{i\in I}w_i\sum_{b\geq 0}\alpha_k^{|I|-1-b}H^b,
\end{align}
and the only poles in the equivariant parameters of $\widehat\bU_\T\left(\one_{(m)}^k\right)_{(-m)}$ occur along $\alpha_k=\alpha_j$ for $j\neq k$, which cancel in the sum. Conditions \textbf{(A1)} and \textbf{(A2)} imply that $H_{(-m)}^{|I|}=0$, showing that \eqref{eq:sumthree} has a well-defined non-equivariant limit. It follows that $\phi_\T^+\left(\bU_\T(\Phi^\mu)_{(-m)}\right)$ has a well-defined non-equivariant limit, concluding the proof of (ii). 

Notice, also, that the non-equivariant limit of $\phi_\T^+\left(\bU_\T(\Phi^\mu)_{(-m)}\right)$ is supported on $H_{(-m)}^{|I|-1}$. Putting this together with \eqref{brvanish}, we see that the non-equivariant limit of  $\phi_\T^+\left(\bU_\T(\Phi^\mu)\right)$ lies in the kernel of $i^*$ whenever $\Phi^{\mu}$ is not narrow. This is important below. 

We now prove assertion (iii). Start by writing $\bff=\bff'+\bff''$ where $\bff'$ is supported on $\H^W$. Then
\begin{align*}
\lim_{\alpha\rightarrow 0}\left(i^*\circ\phi_\T^+\circ\bU_\T\left(\bff \right)\right)&=i^*\left(\lim_{\alpha\rightarrow 0}\left(\phi_\T^+\circ\bU_\T\left(\bff' \right)\right)\right)\\
&=i^*\left(\lim_{\alpha\rightarrow 0}\left(\phi_\T^+\circ\bU_\T\right)\left(\lim_{\alpha\rightarrow 0}\bff' \right)\right)\\
&=\bV\left(\lim_{\alpha\rightarrow 0}\bff \right),
\end{align*}
where the first equality follows from the fact that $\lim_{\alpha\rightarrow 0}\phi_\T^+\circ\bU_\T(\bff'')$ lies in the kernel of $i^*$,  the second follows from the fact that $\phi_\T^+\circ\bU_\T$ has a well-defined non-equivariant limit upon restriction to the narrow subspace, and the third follows from Corollary \ref{cor:nelimit} and the definition of $\bV$. This completes the proof of the lemma, and thus, of Theorem \ref{maintheorem2}.

\end{proof}

%
%
%

\bibliographystyle{abbrv}
\bibliography{biblio}

\end{document}